\documentclass{elsarticle}

\usepackage{amsmath}
\usepackage{amssymb}
\usepackage{latexsym,pifont,color,comment}
\usepackage{algorithm,algorithmic}
\usepackage{amsthm}
\newtheorem{thm}{Theorem}

\def\B{{\bf B}}
\def\half{\frac{1}{2}}
\def\v{{\bf v}}

\def\E{{\bf E}}
\def\J{{\bf J}}
\def\reals{{{\rm l} \kern -.15em {\rm R} }}
\def\xv{{\bf x}}
\def\nv{{\bf n}}

\def\Tm{{\mathcal T}}

\journal{Journal of Computational Physics}

\begin{document}

\begin{frontmatter}

\title{A positivity-preserving
	high-order semi-Lagrangian discontinuous Galerkin scheme for the
	Vlasov-Poisson equations}

\author[author1]{James A. Rossmanith\fnref{labc}}
\ead{rossmani@math.wisc.edu}

\author[author1]{David C. Seal}
\ead{seal@math.wisc.edu}

\address[author1]{Department of Mathematics, University of Wisconsin,
480 Lincoln Drive, Madison, WI  53706-1388, USA}

\fntext[labc]{Corresponding author}

\begin{abstract}
The Vlasov-Poisson equations describe the evolution of a collisionless plasma,
represented through a probability density function (PDF) that self-interacts via an electrostatic
force. One of the main difficulties in numerically solving this system is the severe time-step
restriction that arises from parts of the PDF associated with moderate-to-large velocities.
The dominant approach in the plasma physics community for removing these time-step restrictions
is the so-called particle-in-cell (PIC) method,
which discretizes the distribution function into a set of macro-particles,
while the electric field is represented on a mesh.
Several alternatives to this approach exist, including fully Lagrangian, fully Eulerian, and so-called semi-Lagrangian
methods. The focus of this work is the semi-Lagrangian approach, which begins with a grid-based
Eulerian representation of both the PDF and the electric field, then evolves the PDF via  Lagrangian dynamics,
and finally projects this evolved field back onto the original Eulerian mesh. 
In particular, we develop in this work a method that discretizes the 1+1 Vlasov-Poisson system via  a high-order
discontinuous Galerkin (DG) method in phase space, and an operator split, semi-Lagrangian method in time. 
Second-order accuracy in time is relatively easy to achieve via Strang operator splitting. With additional work,
using  higher-order splitting and a higher-order method of characteristics, we also demonstrate how to push this scheme to fourth-order accuracy
in time. We show how to resolve all of the Lagrangian dynamics in such a way that mass is exactly
conserved, positivity is maintained, 
 and high-order accuracy is achieved.
The Poisson equation is solved to high-order via the smallest stencil local discontinuous Galerkin (LDG)
approach. We test the proposed scheme on several standard test cases.
\end{abstract}


\end{frontmatter}



\section{Introduction}
The Vlasov equation in its various incarnations (e.g.,
Vlasov-Maxwell, Vlasov-Darwin, and
Vlasov-Poisson) models the dynamics of collisionless plasma. 
Plasma is the state of matter where electrons have
dissociated from their nuclei, creating a mixture of
interacting charged particles. This mixture can evolve via a variety of effects,
including electromagnetic interactions 
and through particle-particle collisions. In the
collionless limit, the mean free-path is much larger than 
the characteristic length scale of the plasma; and therefore,
particle-particle collisions are dropped from the mathematical model.
Vlasov models are widely used in both astrophysical applications (e.g., 
\cite{article:BeBh07,article:Bi01,article:SmGr06}), as well
as in laboratory settings (e.g., \cite{christlieb2004eps,article:Vay04, 
article:PaHi97,article:ChHiKe00,article:Idom08}).

The development of accurate and efficient numerical methods for
the solution of the Vlasov equations are faced with a variety
of numerical challenges, the most important of which we describe below.
\begin{itemize}
    \item {\bf High dimensionality.} The Vlasov system is
            a nonlinear and nonlocal advection equation in six
            phase space dimensions (${\bf x}\in \reals^3$ and
             ${\bf v}\in \reals^3$) and time -- this of often referred
             to as $3+3+1$ dimensions. Even though the Vlasov
             equation is in many ways mathematically simpler than 
             fluid models, the fact that it lives in a space of
             twice the number of dimensions makes it computationally
             much more expensive to solve.

\smallskip

  \item {\bf Conservation and positivity.} In fluid models,
  	conservation of mass, momentum, and energy are
	often relatively easy to guarantee in a numerical
	discretization, since each of these quantities is
	a dependent variable of the system. In Vlasov models
	it is generally more difficult to exactly maintain these
	quantities in the numerical discretization. Exact positivity of the 
	probability density function is also not guaranteed by many
	standard discretizations of the Vlasov system; and therefore,
	additional work in choosing the correct approximation spaces
	is often required.	
  	
\smallskip
	
  \item {\bf Small time steps due to ${\bf v} \in \reals^3$.}
  	In the non-relativistic case, the advection
	velocity of the density function in phase
	space depends linearly on the components of the
	velocity vector $\v \in \reals^3$ (see equation 
	\eqref{eqn:vlasov_maxwell} in \S\ref{sec:equations}). 
	Since it is in general possible to
	have
	``particles'' in the Vlasov system that travel arbitrarily fast, 
	 there will be a 
  	severe time-step restriction, relative to the dynamics
	of interest, that arises from parts of the PDF 
	associated with moderate-to-large velocities.	
\end{itemize}

Several approaches have been introduced to try and solve
some of these problems,
including particle-in-cell methods, Lagrangian particle methods,
and grid-based semi-Lagrangian methods. We briefly summarize each
of these approaches below.
\begin{itemize}
\item{\bf Particle-in-cell methods.}  Particle-in-cell (PIC) methods
are ubiquitous in both astrophysical (e.g., \cite{article:Vay04})
and laboratory plasma (e.g., \cite{article:Bi01}) application problems.
The basic approach is outlined in the celebrated textbooks
of Birdsall and Langdon \cite{birdsall2004plasma} and 
Hockney and Eastwood \cite{hockney1988computer}, both
of which appeared in the mid-to-late 1980s. Modern improvements
to these methods are still topics of current research (e.g., 
adaptive mesh refinement \cite{article:Vay04}, very high-order variants
\cite{article:JaHe06,article:JaHe09}, etc$\ldots$).
The basic idea is that the distribution function is discretized
into a set of macro-particles (Lagrangian representation), while the 
electromagnetic field is represented on a mesh (Eulerian representation). 
The main advantage of this approach is that
positivity and mass conservation are essentially automatic,
the small time step restriction is removed due to the
fact that the particles are evolved in a Lagrangian framework,
and the electromagnetic equations can be solved via standard
mesh-based methods.
The main disadvantages of this method are: (1)
numerical
errors are introduced due to the  interpolations
that must de done to exchange information between
 the particles and fields, and (2) error control is non-trivial
 since particles may either cluster or generate rarefied regions
 during the evolution of the plasma.

\smallskip

\item{\bf Lagrangian particle methods.} One possible alternative
to the PIC methodology is to go to a completely Lagrangian
framework -- this removes the need to interpolate between
the particles and fields. Such approaches are commonplace in several
application areas such as 
many body dynamics in astrophysics \cite{article:BarnesHut86},
vortex dynamics  \cite{article:LiKr01}, as well
as in plasma physics  \cite{christlieb2006gfp}. 
The key is that the potential (e.g., gravitational potential, streamfunction, or electric potential) is calculated 
by integrating the point charges
 represented by the Lagrangian particles against a Green's function.
Since the charges are point particles, evaluating this integral reduces to
computing sums over the particles. Naive methods would need
${\mathcal O}(N^2)$ floating point operations to evaluate all of these
sums, where $N$ is the number of particles, 
but fast summation methods such as treecode methods \cite{article:BarnesHut86,article:LiKr01} 
and the fast multipole method \cite{article:GreenRohk87} can be
used to reduce this to ${\mathcal O}(N \log N)$.
The main disadvantage of this approach is that it relies
 on having a Green's function, which for more complicated
dynamics (i.e., full electromagnetism), may be difficult
to obtain.

\smallskip

\item{\bf Semi-Lagrangian grid-based methods.}
Another alternative to PIC is to switch to a completely
grid-based method. Such an approach allows for
a variety of high-order spatial discretizations, and can
be evolved forward in time via so-called {\it semi-Lagrangian}
time-stepping. 
The basic idea is that the PDF sits initially on a grid; the PDF is then
evolved forward in time using Lagrangian dynamics; and finally, 
the new PDF is
projected back onto the original mesh. This gives many of the
advantages of particle methods (i.e., no small time-step
restrictions), but retains a nice grid structure for both
the PDF and the fields, allowing extension to
very high-order accuracy. There have been several
contributions to this approach over the last few years. One of the
first papers that developed a viable semi-Lagrangian method
was put forward by Cheng and Knorr \cite{article:ChKn76}.
More recent activity on this approach includes the work of
Parker and Hitchon \cite{article:PaHi97},
Sonnendr\"ucker and his collaborators (see for example \cite{article:CrMeSo10,article:FiSo03,article:SoRoBeGh99, article:CoSoDiBeGh99, article:BeSeSo05,article:CrLaSo07,article:CrReSo09,article:BeCrDeSo09}),
and Christlieb and Qiu \cite{article:QiuCh10}.

\end{itemize}

The goal of the current work is to develop a high-order,
grid-based, semi-Lagrangian method for solving the 1+1 Vlasov-Poisson
equation. We focus on the 1+1 case, leaving the problem of
high-dimensionality for future work. Our discretization is based on high-order discontinuous Galerkin  representations and a high-order operator split semi-Lagrangian
time-stepping method. We argue that this approach is a promising method
that produces very accurate results at relatively low computational
expense. 
 
This paper begins with a brief review of the
Vlasov equations in \S \ref{sec:equations}.
Part of the focus of the present
work is to develop a higher-order version of the classical
Cheng and Knorr \cite{article:ChKn76} operator splitting
method, which we review in \S \ref{sec:cheng_knorr}. The spatial
discretization for the proposed method will be based on the
discontinuous Galerkin method, which we briefly review in
\S \ref{sec:dgfem}. The heart of this paper is \S \ref{sec:numericaL_method}, which
details all the aspects of the proposed method, and in particular,
explains how to achieve high-order in space and time, mass conservation,
and positivity of the distribution function. Finally in \S 
\ref{sec:numericaL_examples} we apply the proposed scheme to a
variety of standard test cases for the Vlasov-Poisson system, including
the two-stream instability problem and Landau damping.

\section{Mathematical equations}
\label{sec:equations}
The Vlasov system
describes the evolution of a probability density function (PDF) in phase space:
\begin{equation}
	f_s(t, {\bf x}, {\bf v}): \reals^{+} \times \reals^d \times \reals^d
	\rightarrow \reals^S,
\end{equation}	
where $d$ is the spatial dimension and $S$ represents the
number of plasma species. This PDF
denotes the probability of finding a particle of species $s$ at time $t$, at location
${\bf x}$, and with velocity ${\bf v}$. Although the  PDF is not itself a physical observable, its moments represent various physically observable quantities:
\begin{alignat}{2}
\rho_s(t,{\bf x}) &:= \int_{\reals^d}  \, f_s \, d{\bf v}, \qquad && \text{(mass density of species $s$)}, \\
	\rho_s {\bf u}_s(t,{\bf x}) &:=  \int_{\reals^d} \, {\bf v} \, f_s \, d{\bf v}, 
	\qquad && \text{(momentum density of species $s$)},  \\
	{\mathcal E}_s(t,{\bf x}) &:=  \frac{1}{2} \int_{\reals^d}  \, \| {\bf v} \|^2 \, f_s \, d{\bf v},
	 \qquad && \text{(energy density of species $s$)}.
\end{alignat}

Under the assumptions of a non-relativistic and collisionless plasma,
the PDF for each species obeys the Vlasov equation, which is
an advection equation in $({\bf x}, {\bf v})$ phase space:
  \begin{equation}
  \label{eqn:vlasov_maxwell}
\frac{\partial f_s}{\partial t} + \v \cdot \nabla_{\xv}
  	f_s + \frac{q_s}{m_s} \left( \E + \v \times \B \right)
	 \cdot \nabla_{\v} f_s = 0.
\end{equation}
The ``particles'' represented by this kinetic description do not
interact through collisional processes; and instead, are only coupled
indirectly through the electromagnetic field. In general, the
electromagnetic field satisfies Maxwell's equations:
\begin{gather}
	 \frac{\partial}{\partial t}
	 \begin{bmatrix}
	 	\B \\ \E
	 \end{bmatrix} +	
	 \nabla \times
	 \begin{bmatrix}
	    \E \\ -c^2 \B
	 \end{bmatrix} = 
	 \begin{bmatrix}
	 	0 \\ -c^2 \J
	 \end{bmatrix}, \\
	 \nabla \cdot \B = 0, \quad
	 \nabla \cdot \E = c^2 \sigma,
\end{gather}
where ${\bf B}$ is the magnetic field, ${\bf E}$ is the electric field, and
the total charge density and total current densities
are given by the following:
\begin{equation}
	 \sigma = \sum_s \frac{q_s}{m_s} \rho_s, \quad
	 \J = \sum_s \frac{q_s}{m_s} \rho_s {\bf u}_s.
\end{equation}
Note that the electromagnetic field variables, $\B(t,{\bf x})$ and
$\E(t,{\bf x})$, as well as the mass and momentum densities,
$\rho_s(t,{\bf x})$ and $\rho_s {\bf u}_s(t,{\bf x})$, only depend
on time and the spatial coordinates ${\bf x}$. 
  
In this work we will not consider the full Vlasov-Maxwell system
for a many species plasma; and instead, we only consider the
single-species Vlasov-Poisson equation. To arrive at the Vlasov-Poisson
system, we start with Vlasov-Maxwell and 
assume that the charges are slow-moving
in comparison to the speed of light; this allows us to replace the full
electromagnetic equations with electrostatics. Furthermore, 
we consider only two-species: one dynamically evolving
species, which we take without loss of generality
to have positive charge and unit mass $m=1$, and one
stationary background species that has a charge of opposite sign to the dynamic
species.
Because the background charge is stationary, we will only need
to solve a single-species Vlasov equation. These assumptions conspire
to form the Vlasov-Poisson equations:
\begin{gather}
   \label{eqn:VlasovPoisson}
   	f_{,t}  +  {\bf v} \cdot f_{, {\bf x}} + \nabla \phi \cdot  f_{, {\bf v}} = 0, \\
		\nabla^2 \phi =  \rho(t,{\bf x}) - \rho_0,
\end{gather}
where $\phi$ is the electric potential: $\nabla \phi = {\bf E}$, and
$-\rho_0$ is the stationary background charge density. 

The Vlasov-Poisson system contains an infinite number of
quantities that are conserved in time. These can be used
as diagnostics in a numerical discretization. We list four quantities that
will be used in diagnosing our proposed scheme, all of which should
remain constant in time:
\begin{align}
\label{eqn:L1norm}
	\| f \|_{L_1} &:=\int_{\reals^d} \int_{\reals^d}   \bigl| f \bigr| \, d {\bf v} \, d{\bf x}, \\
\label{eqn:L2norm}
	\| f \|_{L_2} &:= \left( {\int_{\reals^d} \int_{\reals^d}    f^2 \, d {\bf v} \, d{\bf x}} \right)^{\frac{1}{2}}, \\	
\label{eqn:energy}
	\text{Total energy} &:=  \frac{1}{2} 
		\int_{\reals^d} \int_{\reals^d}  \| {\bf v} \|^2 f \, d {\bf v} \, d{\bf x}
		 +  \frac{1}{2} \int_{\reals^d}  \| {\bf E} \|^2  \, d{\bf x}, \\
\label{eqn:entropy}
	\text{Entropy} &:= -\int_{\reals^d} \int_{\reals^d}   f \, \log(f) \, d {\bf v} \, d{\bf x}.
\end{align}

Finally, we point out that in the current work we are concerned
exclusively with the 1+1 dimensional version of the above equations
with periodic boundary conditions in $x$. In this case, the Vlasov-Poisson
system on $\Omega = (t,x,v) \in \reals^+ \times [-L,L] \times \reals$ is:
\begin{gather}
\label{eqn:VP1d}
	f_{,t} + v f_{,x} + E(t,x) \, f_{,v} = 0, \\
\label{eqn:VP1dp}
	E_{,x} = \rho(t,x) - \rho_0,
\end{gather}
with periodic boundary conditions:
\begin{align*}
	f(t,-L,v) = f(t,L,v), \quad E(t,-L) = E(t,L), \quad \text{and} \quad
		\phi(t,-L) = \phi(t,L).
\end{align*}
In these expressions we used the shorthand notation:
\[
	x := x^1, \quad v := v^1, \quad \text{and} \quad E := E^1.
\]
The total and background densities are 
\begin{gather}
	\rho(t,x) := \int_{-\infty}^{\infty} f(t,x,v) \, dv \qquad \text{and}
	\qquad
	\rho_0 := \frac{1}{2L} \int_{-L}^{L} \rho(t,x) \, dx.
\end{gather}
Note that $\rho_0$ is in fact constant in time, due to
conservation of mass on the periodic domain $[-L,L]$.

\section{Strang operator splitting}
\label{sec:cheng_knorr}
If we momentarily freeze the electric field in time, 
the Vlasov equation \eqref{eqn:VlasovPoisson} can be viewed
as an advection equation of the following form:
\begin{equation}
	f_{,t} + {\bf a}(\v) \cdot f_{,\xv} + {\bf b}(\xv) \cdot f_{,\v} = 0.
\end{equation}
Cheng and Knorr \cite{article:ChKn76} realized that such an equation
can be handled very efficiently if split into the following
two sub-problems:
\begin{align*}
   \text{Problem ${\mathcal A}$:} & \quad f_{,t} + {\bf a}(\v) \cdot f_{,\xv} = 0, \\
   \text{Problem ${\mathcal B}$:} & \quad f_{,t} + {\bf b}(\xv) \cdot f_{,\v} = 0.
\end{align*}
The key benefit of this splitting is that each operator is now a constant
coefficient advection equation (i.e., the transverse coordinate acts only as
a parameter),  each of which can be handled very simply
with a variety of spatial discretization and semi-Lagrangian time-stepping.
The down side of this approach,
of course, is the introduction of splitting error.

Cheng and Knorr \cite{article:ChKn76} developed a
second order accurate version of this scheme 
via Strang operator splitting \cite{article:St68}.
Their scheme is summarized in Algorithm \ref{alg1}.
It is worth pointing out that the electric field computed in 
Step 2, $\E^{n+\frac{1}{2}}$, is second order accurate in time,
even though it is computed after advection in the ${\bf x}$ variables only.

\noindent {\bf Claim.}
{\it Assuming that the current solution at time $t=t^n$ is known exactly,
and that each step in Algorithm \ref{alg1} is carried out exactly
in space, velocity, and time, the density computed in Step 2 is
second order accurate in time:
\[
	\rho^{n+\frac{1}{2}} = \rho\left(t^n + \frac{\Delta t}{2}, {\bf x} \right) + {\mathcal O}\left(
		\Delta t^2 \right).
\]
This also implies that the electric field in Step 2 is 
second order accurate in time:
\[
	\E^{n+\frac{1}{2}} = \E\left(t^n + \frac{\Delta t}{2}, {\bf x} \right) + {\mathcal O}\left(
		\Delta t^2 \right).
\]
}

\begin{proof}
By assumption the PDF after the first step satisfies the following relationship:
\[
	\tilde{f}\left({\bf x},{\bf v}\right) := f\left(t^n, \xv - \frac{\Delta t}{2} \v, \v \right).
\]
We integrate this relationship in velocity to compute the density at time 
$t^n + \frac{\Delta t}{2}$:
\[
	\begin{split}
	\rho^{n+\frac{1}{2}} &:= \int_{\bf v} \tilde{f}\left({\bf x},{\bf v}\right) d{\bf v}
		=  \int_{\bf v}  f\left(t^n, \xv - \frac{\Delta t}{2} \v, \v \right) d{\bf v} \\
		&=  \int_{\bf v}  f\left(t^n, \xv, \v \right) d{\bf v} 
			- \frac{\Delta t}{2} \nabla_{\xv} \cdot \left\{
			\int_{\bf v}  {\bf v} f\left(t^n, \xv, \v \right) d{\bf v} \right\} +
		{\mathcal O}(\Delta t^2) \\ &=
		\rho^n - \frac{\Delta t}{2} \nabla_{\xv} \cdot \left( \rho^n \, {\bf u}^n \right)
			+ {\mathcal O}(\Delta t^2).
		\end{split}
\]
Finally, we use the fact that 
\[
	\rho^n_{,t} = - \nabla_{\bf x} \left(\rho^n {\bf u}^n \right),
\]
in order to assert that
\[
	\rho^{n+\frac{1}{2}} = \rho^n + \frac{\Delta t}{2} \rho^n_{,t}
			+ {\mathcal O}(\Delta t^2)=
				\rho\left(t+\frac{\Delta t}{2},{\bf x} \right)+ {\mathcal O}(\Delta t^2),
\]
which proves the claim.
\end{proof}

\begin{algorithm}
\caption{Cheng and Knorr \cite{article:ChKn76} operator split algorithm.}
\label{alg1}
\begin{algorithmic}
  \STATE 1. $\frac{1}{2}\Delta t$ \quad step on \quad $f_{,t} + {\bf v} \cdot f_{,{\bf x}} = 0$.
  \STATE
   \STATE 2. Solve \quad $\nabla^2 \phi = \rho^{n+\frac{1}{2}} -  \rho_0$,
  	\quad and  compute \quad ${\bf E}^{n+\frac{1}{2}} = \nabla \phi$.
 \STATE
   \STATE 3. $\Delta t$ \quad step on \quad $f_{,t} + {\bf E}^{n+\frac{1}{2}} \cdot f_{,{\bf v}} = 0$.
   \STATE
   \STATE 4. $\frac{1}{2} \Delta t$ \quad step on \quad $f_{,t} + {\bf v} \cdot f_{,{\bf x}} = 0$.
\end{algorithmic}
\end{algorithm}

The original method of Cheng and Knorr \cite{article:ChKn76} employs
a cubic spline spatial discretization. In the past few years, work on semi-Lagrangian
solvers using Algorithm \ref{alg1} with a variety of spatial discretizations, including
modified cubic spline interpolants, has been carried out by
 Sonnendr\"ucker and his collaborators (see for example \cite{article:CrMeSo10,article:FiSo03,article:SoRoBeGh99, article:CoSoDiBeGh99, article:BeSeSo05,article:CrLaSo07,article:CrReSo09,article:BeCrDeSo09}).
Another recent contribution to this approach was the method of Christlieb and Qiu \cite{article:QiuCh10}, who combined the operator split method with high-order
 WENO
(weighted essentially non-oscillatory) finite differences for the spatial derivatives.

The focus of the present work is to again consider operator splitting techniques,
but this time with high-order discontinuous Galerkin (DG) method for the spatial discretization,
and with fourth-order operator splitting techniques. 
We describe the basic DG framework in the next section
\S \ref{sec:dgfem} and various details of both a second and a fourth-order
in time operator split approach in \S \ref{sec:numericaL_method}.

\section{The discontinuous Galerkin (DG) method}
\label{sec:dgfem}
The modern form of the discontinuous Galerkin (DG) method
was developed in a series of papers by Bernardo Cockburn, 
Chi-Wang Shu, and their collaborators \cite{article:CoShu1,article:CoShu2,article:CoShu3,article:CoShu4,article:CoShu5}. 
In this section we briefly
review the DG method for a general two-dimensional conservation
law on a Cartesian mesh. This section will also serve to introduce
the notation that we will use throughout this paper.

Consider a general 2D conservation
law of the form:
\begin{equation}
\label{eqn:conslaw}
q_{,t} +  f(q,t,\xv)_{,x} + g(q,t,\xv)_{,y} = 0,
\quad \text{in} \quad \xv \in \Omega \subset \reals^2,
\end{equation}
with appropriate initial and boundary conditions.
In this equation $q(t,\xv) \in \reals^m$ is the vector of conserved variables and
$f(q,t,\xv), g(q,t,\xv) \in \reals^m$ are the flux functions in the $x$ and $y$-directions,
respectively.
 We assume that equation (\ref{eqn:conslaw}) is hyperbolic, 
 meaning that the family of $m\times m$ matrices defined by
 \begin{equation}
   \label{eqn:fluxJacobian}
    A(q,\xv; \nv) = \nv \cdot \left( 
    \frac{\partial f}{\partial q}, \, 
    \frac{\partial g}{\partial q} \right)^{T}
\end{equation}
are diagonalizable with real eigenvalues for
all $\xv$ and $q$ in the domain of interest and
for all $\| \nv \| = 1$.

We construct a Cartesian grid over $\Omega = 
[a_x, \, b_x] \times [a_y, \, b_y]$,
with uniform grid spacing $\Delta x$ and $\Delta y$ in each
coordinate direction. The mesh elements are centered 
at the coordinates
\begin{equation}
          x_i =  a_x  +  \left( i - \frac{1}{2} \right)  \Delta x
          \quad \text{and} \quad
          y_j =  a_y  +  \left( j - \frac{1}{2} \right) \Delta y.
\end{equation}
On this grid we define the {\it broken} finite element space
\begin{equation}
\label{eqn:broken_space}
    W^h = \left\{ w^h \in L_{\infty}(\Omega): \,
    w^h |_{\Tm} \in P^{q}, \, \forall \Tm \in \Tm_h \right\},
\end{equation}
where $W^{h}$ is shorthand notation for
$W^{\Delta x, \Delta y}$.
The above expression means that on each element $\Tm$, $w^h$ will
be a polynomial of degree at most $q$, and no continuity is assumed
across element edges.
Each element can be mapped to the canonical element $(\xi, \eta) \in [-1,1] \times [-1,1]$ via the linear transformation:
\begin{equation}
	x = x_i + \xi \, \frac{\Delta x}{2}, \quad 
	y = y_j + \eta \, \frac{\Delta y}{2}.
\end{equation}
The normalized Legendre
polynomials up to degree four on the canonical element can be written as
\begin{equation*}
\begin{split}
\varphi^{(\ell)} = \Biggl\{ &1, \, \, \,  \sqrt{3} \, \xi, \, \, \,  \sqrt{3} \, \eta, 
  \, \, \,   3 \, \xi \eta, \, \, \,  \frac{\sqrt{5}}{2} \left( 3 \xi^2 - 1 \right),
\, \, \,  \frac{\sqrt{5}}{2} \left( 3 \eta^2 - 1 \right), \\ 
&\frac{\sqrt{15}}{2} \eta \, (3 \xi^2 - 1), \, \, \,
\frac{\sqrt{15}}{2} \xi \, (3 \eta^2 - 1), \, \, \,
\frac{\sqrt{7}}{2} (5 \xi^3 - 3 \xi), \, \, \,
\frac{\sqrt{7}}{2} (5 \eta^3 - 3 \eta), \\
&\frac{\sqrt{21}}{2} \eta \, (5 \xi^3 - 3 \xi), \, \, \,
\frac{\sqrt{21}}{2} \xi \, (5 \eta^3 - 3 \eta), \, \, \,
\frac{5}{4} (3 \xi^2 - 1) (3 \eta^2 - 1), \\
&\frac{105}{8} \xi^4  - \frac{45}{4} \xi^2  + \frac{9}{8}, \, \, \,
\frac{105}{8} \eta^4  - \frac{45}{4} \eta^2  + \frac{9}{8} \Biggr\}.
\end{split}
\end{equation*}
These basis functions are orthonormal with respect
to the following inner product:
\begin{equation}
   \Bigl\langle \varphi^{(m)}, \, \varphi^{(n)} \Bigr\rangle :=  
   \frac{1}{4}  \int_{-1}^{1} \int_{-1}^{1}
    \varphi^{(m)}(\xi,\eta) \, \varphi^{(n)}(\xi,\eta) \, d\xi \, d\eta = 
     \delta_{mn}.
\end{equation}
We will look for approximate solutions of (\ref{eqn:conslaw})
that have the following form:
\begin{equation}
\label{eqn:q_ansatz}
q^h(t, \xi, \eta) \Bigl|_{\Tm_{ij}}  :=
  \sum_{k=1}^{M(M+1)/2}  Q^{(k)}_{ij}(t) \, \varphi^{(k)}(\xi, \eta),
\end{equation}
where $M$ is the desired order of accuracy in space (i.e., for fifth
order: $M=5$ and $M(M+1)/2 = 15$).
The Legendre coefficients of the initial conditions at $t=0$ are 
determined from the $L_2$-projection of $q^h(x,y,0)$ onto the Legendre basis functions:
\begin{align}
\label{eqn:l2project}
Q^{(k)}_{ij}(0)  &:= \Bigl\langle q^h(0, \xi, \eta), \, \varphi^{(k)}(\xi,\eta) \Bigr\rangle.
\end{align}
In practice, these double integrals are evaluated using
standard 2D Gaussian quadrature rules involving $M^2$ points. 

In order to determine the Legendre coeficients for $t>0$,
we multiply conservation law (\ref{eqn:conslaw}) by the test function $\varphi^{(\ell)}$
and integrate over the grid cell $\Tm_{ij}$.
After the appropriate integrations-by-part, 
we arrive at the following semi-discrete evolution equations:
\begin{equation}
\label{eqn:semidiscrete}
\frac{d}{dt} \,
   Q^{(\ell)}_{ij} = {\mathcal L}^{(\ell)}_{ij}(Q,t) := N^{(\ell)}_{ij}
- \frac{\Delta {\mathcal F}_{ij}^{(\ell)}}{\Delta x}
- \frac{\Delta {\mathcal G}_{ij}^{(\ell)}}{\Delta y} ,
\end{equation}
where 
\begin{align}
	\label{eqn:Nvals}
	N^{(\ell)}_{ij} &= \frac{1}{2}
   \int_{-1}^{1} \int_{-1}^{1}
   \left[ \, \frac{1}{\Delta x} \, \varphi^{(\ell)}_{, \xi} \, f(q^h,t,\xv) 
   	+ \frac{1}{\Delta y} \, \varphi^{(\ell)}_{, \eta} \, g(q^h,t,\xv) \, \right] \, 
   d\xi  \, 
   d\eta, \\
	\label{eqn:Fl1}
	\Delta {\mathcal F}^{(\ell)}_{ij} &= 
	\left[ \frac{1}{2} \int_{-1}^{1}
	\varphi^{(\ell)} \,
		f(q^h,t,\xv)  \, 
		 d\eta \right]_{\xi=-1}^{\xi=1},  \\ 
		 \label{eqn:Gl1}
		\Delta {\mathcal G}^{(\ell)}_{ij} &= \left[
		 \frac{1}{2} \int_{-1}^{1}
	\varphi^{(\ell)} \,
		g(q^h,t,\xv)  \, d\xi
		 \right]_{\eta=-1}^{\eta=1}.
\end{align}
The integrals in (\ref{eqn:Nvals}) can
be numerically approximated via standard 
2D Gaussian quadrature rules involving $(M-1)^2$ points. 
The integrals in (\ref{eqn:Fl1}) and  (\ref{eqn:Gl1}) can
be approximated with standard 1D Gauss quadrature
rules involving $M$ points.

\section{A high-order semi-Lagrangian DG method}
\label{sec:numericaL_method}
We describe in this section a semi-Lagrangian discontinuous
Galerkin method for solving the Vlasov-Poisson system.
This method will have all of the following properties:
\begin{enumerate}
\item {\it Unconditionally stable};
 \item {\it High-order accurate in space} ($5^{\text{th}}$ order);
 \item {\it High-order accurate in time} ($4^{\text{th}}$ order);
\item {\it Mass conservative}; and
\item {\it Positivity-preserving}.
\end{enumerate}
All of these properties are explained in detail in this section.

We begin by explaining the basic
idea on the constant coefficient 1D advection equation in
\S \ref{sec:1dadv}. The extension of this 1D scheme
to the 1+1 Vlasov equation via Strang operator splitting
is described in section
\S \ref{sec:vlasov_solver}. A simple and efficient local
discontinuous Galerkin solver for the Poisson equation 
is described in \S \ref{sec:poisson}. The
generalization of the Strang-split scheme to higher-order
splitting is shown in \S \ref{sec:higherordersplit}.
Basic properties including
conservation of mass and positivity in the mean are
proved in \S \ref{sec:cons_mass}. Finally, in \S \ref{sec:positivity}
a limiter that provides global pointwise positivity is described.

\begin{figure}[!t]
\begin{center}
   (a) \, \includegraphics[width=48mm]{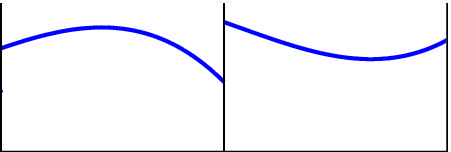} \qquad
   (b) \, \includegraphics[width=48mm]{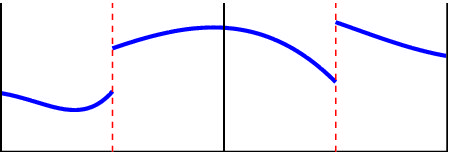} 
   
   \vspace{4mm}
   
   (c) \, \includegraphics[width=48mm]{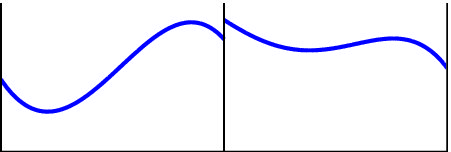}
  \caption{Illustration of the shift + project method for solving
  	the constant coefficient advection equation in 1D as 
	described in \S \ref{sec:1dadv}. Panel (a)
	shows piecewise polynomial initial data; Panel (b) shows
	the initial data shifted by some amount (i.e., the exact evolution of
	the initial data); and finally, Panel (c) shows the
	solution after it has been re-projected back onto the
	original piecewise polynomial basis.\label{fig:shift_proj}}
\end{center}
\end{figure}

\subsection{A toy problem: the 1D advection equation}
\label{sec:1dadv}
Consider the 1D constant coefficient advection equation:
\begin{equation}
	f_{,t} + v f_{,x} = 0, \quad (t,x) \in \reals^{+} \times \reals,
\end{equation}
with initial condition $f(0,x)$. For simplicity of exposition, let's assume in
the discussion below that that $v>0$; the extension to the case $v<0$ is
straightforward. We consider solving
this equation on a uniform mesh of elements, $\Tm_i$, that 
each have width $\Delta x$. We begin by projecting the initial condition
onto the mesh:
\begin{equation}
	F^h(0,\xi) \Bigl|_{\Tm_i} = \sum_{\ell=1}^M  F^{(\ell)}_i(0) \, 
	\varphi^{(\ell)}_{\text{1D}}(\xi),
\end{equation}
where $F^h$ represents the finite dimensional approximation of
$f(t,x)$, $M$ is the desired order of accuracy, and  $\varphi_{\text{1D}}^{(\ell)}(\xi)$ are the
1D   Legendre basis functions:
\begin{equation*}
\begin{split}
\varphi^{(\ell)}_{\text{1D}} = \Biggl\{ &1, \, \, \,  \sqrt{3} \, \xi,
 \, \, \,  \frac{\sqrt{5}}{2} \left( 3 \xi^2 - 1 \right), \, \, \,
\frac{\sqrt{7}}{2} (5 \xi^3 - 3 \xi), \, \, \,
\frac{3}{8} \left( 35 \xi^4  - 30 \xi^2  + 3 \right) \Biggr\}.
\end{split}
\end{equation*}
A simple, high-order accurate,
and unconditionally stable algorithm to update this solution can developed
based on the following two steps:
\begin{enumerate}
\item  Exactly advect the initial condition over a time step $\Delta t$:
\[
    f \left( t+\Delta t, x \right) = f\left(t, x - v \Delta t \right)
\]
\item Project this solution back onto the mesh $\Tm_i$.
\end{enumerate}
This process is illustrated in Figure \ref{fig:shift_proj}.
These two steps can be compactly written for any starting time $t^n$ and
final time $t^{n+1} = t^n + \Delta t$ as follows:
\begin{equation}
\label{eqn:LxW1d}
\begin{split}
	F^{(\ell)}_i \bigl(t^{n+1} \bigr) &= \frac{1}{2} \sum_{k=1}^M F^{(k)}_{i-1-j}(t^n)
	\int_{-1}^{-1 + 2 \nu} \varphi^{(k)}_{\text{1D}}(\xi+2-2\nu)  \, \varphi^{(\ell)}_{\text{1D}}(\xi) \, d\xi  \\
	&+ \frac{1}{2} \sum_{k=1}^M F^{(k)}_{i-j}(t^n)
	\int_{-1+2\nu}^{1}  \varphi^{(k)}_{\text{1D}}(\xi-2\nu)   \, \varphi_{\text{1D}}^{(\ell)}(\xi) \, d\xi,
\end{split}
\end{equation}
where
\begin{align}
j:= \Bigg\lfloor \frac{v \Delta t}{\Delta x} \Bigg\rfloor \qquad
\text{and} \qquad
\nu := \frac{v \Delta t}{\Delta x} - j.
\end{align}
Here $\lfloor \cdot \rfloor$ denotes the {\it floor} operation\footnote{this function
takes a real input and rounds down to the largest integer that is smaller
than or equal to the input.} and  $0 \le \nu \le 1$.
By construction, update \eqref{eqn:LxW1d} is {\it unconditionally stable} independent
of the polynomial order of the spatial discretization.

The integrals in equation \eqref{eqn:LxW1d} can be evaluated exactly.
For example, in the case of piecewise constants and $j=0$, \eqref{eqn:LxW1d} 
is nothing more than the first-order upwind scheme:
\begin{equation}
	F^{(1),n+1}_i  = 
		F^{(1), n}_i  - \nu \left( F^{(1), n}_i
		- F^{(1), n}_{i-1} \right).
\end{equation}
In the case of piecewise linear polynomials and $j=0$, the scheme can be written
as follows:
\begin{align}
\begin{split}
\label{eqn:modLxW1}
F^{(1),n+1}_i &= F^{(1),n}_i - \nu \left( \left[ F^{(1),n}_i+\sqrt{3} \, F^{(2),n}_i \right] -\left[ F^{(1),n}_{i-1}+\sqrt{3} \, F^{(2),n}_{i-1} \right] \right) \\  & \qquad 
\qquad + \sqrt{3} \, \nu^2 \left( F^{(2),n}_i
-F^{(2),n}_{i-1} \right),
\end{split} \\
\begin{split}
\label{eqn:modLxW2}
F^{(2),n+1}_i &= F^{(2),n}_i + \sqrt{3} \,  \nu \left( \left[ F^{(1),n}_i - \sqrt{3} \,  
F^{(2),n}_i \right] -
  \left[  F^{(1),n}_{i-1} +\sqrt{3} \, F^{(2),n}_{i-1} \right] \right) \\ & \hspace{-3mm}
- \sqrt{3} \,  \nu^2 \left( F^{(1),n}_{i}-F^{(1),n}_{i-1}-2 \sqrt{3} \, F^{(2),n}_{i-1} \right) + 2 \nu^3 \,  \left( 
F^{(2),n}_{i}-F^{(2),n}_{i-1} \right).
\end{split}
\end{align}
The above method is a close cousin to the Lax-Wendroff discontinuous Galerkin
scheme (LxW-DG) of Qiu, Dumbser, and Shu \cite{article:Qiu05}. 
In particular, in the case of piecewise linear polynomials the LxW-DG for the
advection equation can be written as
\begin{align}
\begin{split}
\label{eqn:LxW1}
F^{(1),\text{n+1}}_i &= F^{(1),n}_i - \nu \left( \left[ F^{(1),n}_i+\sqrt{3} \, F^{(2),n}_i \right] -\left[ F^{(1),n}_{i-1}+\sqrt{3} \, F^{(2),n}_{i-1} \right] \right) \\  & \qquad 
\qquad + \sqrt{3} \, \nu^2 \left( F^{(2),n}_i
-F^{(2),n}_{i-1} \right),
\end{split} \\
\begin{split}
\label{eqn:LxW2}
F^{(2),\text{n+1}}_i &= F^{(2),n}_i + \sqrt{3} \,  \nu \left( \left[ F^{(1),n}_i - \sqrt{3} \,  
F^{(2),n}_i \right] -
  \left[  F^{(1),n}_{i-1} +\sqrt{3} \, F^{(2),n}_{i-1} \right] \right) \\ &
- 3 \,  \nu^2 \left( F^{(2),n}_{i}-F^{(2),n}_{i-1} \right).
\end{split}
\end{align}
We note that \eqref{eqn:LxW1}--\eqref{eqn:LxW2} and
 \eqref{eqn:modLxW1}--\eqref{eqn:modLxW2} agree
 except in the $\nu^2$ and $\nu^3$ terms in the $F^{(2)}$ update.
We argue below that these additional
terms in the method given by  \eqref{eqn:modLxW1}--\eqref{eqn:modLxW2} 
are crucial in ensuring stability up to CFL number one, and 
their absence in the LxW-DG method cause a non-optimal
stability result.
\bigskip

\noindent {\bf Claim.}
{\it The numerical update given by \eqref{eqn:modLxW1}--\eqref{eqn:modLxW2}
 is stable for $0 \le \nu \le 1$.}

\begin{proof}
We apply von Neumann stability analysis to 
\eqref{eqn:modLxW1}--\eqref{eqn:modLxW2} by assuming the following ansatz:
\[
F^{(\ell),n}_i = F^{(\ell),n} \, e^{I \xi i \Delta x},
\]
where $I = \sqrt{-1}$.
Plugging this into \eqref{eqn:modLxW1}--\eqref{eqn:modLxW2} yields
\[
\begin{bmatrix}
F^{(1)} \\
F^{(2)}
\end{bmatrix}^{n+1}
\hspace{-2mm}
=
\begin{bmatrix}
1+\nu \left( \zeta - 1 \right) &  \sqrt{3} \nu \left(1-\nu \right) \left( \zeta - 1 \right)  \\
\sqrt{3} \nu (\nu-1)(\zeta-1) & 1+2 \nu^3(1-\zeta) + 6 \nu^2 \zeta - 3\nu(\zeta+1)
\end{bmatrix}
\begin{bmatrix}
F^{(1)} \\
F^{(2)}
\end{bmatrix}^n,
\]
where $\zeta = e^{-I \xi \Delta x}$.
With some work, which is omitted here, one can show that the
maximum modulus of the eigenvalues of the amplification matrix
for $0 \le \nu \le 1$ is given by
\[
	g(\nu) = \text{max}\biggl( 1, \Bigl| 1-6 \nu+6 \nu^2 \Bigr| \biggr).
\]
We note that
\[ 
g(\nu) \equiv 1 \quad \forall \, \nu \in [0,1],
\]
which concludes the proof.
\end{proof}

\bigskip

\noindent {\bf Claim.}
{\it The numerical update given by
 \eqref{eqn:LxW1}--\eqref{eqn:LxW2} is stable for $0 \le \nu \le \frac{1}{3}$.}

\begin{proof}
Using the same von Neumann ansatz as in the previous claim,
this time applied to \eqref{eqn:LxW1}--\eqref{eqn:LxW2}, yields 
\[
\begin{bmatrix}
F^{(1)} \\
F^{(2)}
\end{bmatrix}^{n+1}
=
\begin{bmatrix}
1+\nu \left( \zeta - 1 \right) &  \sqrt{3} \nu \left(1-\nu \right) \left( \zeta - 1 \right)  \\
\sqrt{3} \nu (1-\zeta) & 1-3 \nu (\zeta+1) + 3\nu^2 (\zeta-1)
\end{bmatrix}
\begin{bmatrix}
F^{(1)} \\
F^{(2)}
\end{bmatrix}^n.
\]
With some work, which is omitted here, one can show that the
maximum modulus of the eigenvalues of the amplification matrix
for $0 \le \nu \le 1$ is given by
\[
	g_{\text{LxW-DG}}(\nu) = \text{max}\biggl( 1, \Bigl| 1-6 \nu \Bigr| \biggr).
\]
We note that
\[ 
g_{\text{LxW-DG}}(\nu) \equiv 1 \quad \forall \, \nu \in \biggl[0,\frac{1}{3} \biggr],
\quad \text{but} \quad
g_{\text{LxW-DG}}(\nu) > 1 \quad \forall \, \nu \in \biggl(\frac{1}{3}, 1\biggr],
\]
which concludes the proof.
Also note that $g_{\text{LxW-DG}}(\nu)$ and $g(\nu)$ are the same
except for the additional term $6 \nu^2$ in $g(\nu)$. Therefore, we
conclude that this additional
term is crucial in ensuring stability up to CFL number one for the
method given by \eqref{eqn:modLxW1}--\eqref{eqn:modLxW2}.
\end{proof}

In the above discussion we focused on the piecewise linear
DG method. However, more generally, the LxW-DG method
behaves similarly to the standard Runge-Kutta DG methods (RK-DG) \cite{article:CoShu1}
 in that the
maximum allowable CFL number is inversely proportional to the polynomial order
of the spatial discretization.  The modified LxW-DG scheme represented by 
\eqref{eqn:modLxW1}--\eqref{eqn:modLxW2}, and more generally by 
\eqref{eqn:LxW1d}, always has some modified and some
additional terms in the update (e.g., the $\nu^2$ and $\nu^3$ terms
in \eqref{eqn:modLxW2}) that produce a scheme with optimal stability.

\begin{figure}[!t]
\begin{center}
  (a) \includegraphics[height=33mm]{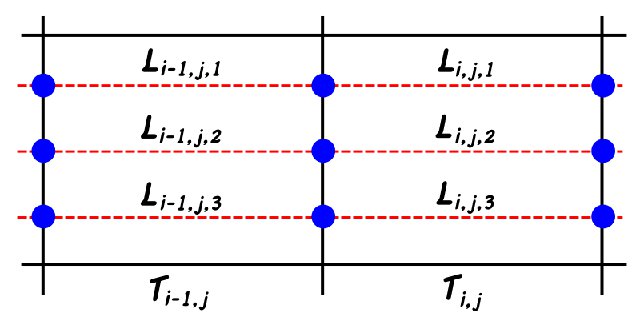} \,
  (b) \includegraphics[height=33mm]{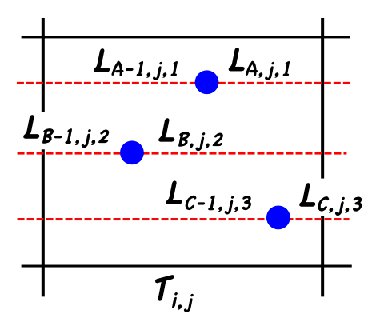}
   \caption{Illustration of the advection algorithm used
   	in the proposed semi-Lagrangian method. Panel (a) illustrates
	a method with 3 Gaussian quadrature lines in the $x$-direction.
	The solution along  each line segment, e.g., $L_{i,j,1}$, in each element, e.g., ${\mathcal T}_{ij}$, is a 1D polynomial.
	Each Gaussian quadrature line has a different velocity, and as
	such, each line will get shifted by a different amount; this is shown in Panel (b).  The subscript labels $A$, $B$, and $C$ in Panel (b) highlight the
	fact that after advection in the $x$-direction, each Gaussian
	quadrature line might come from a different elements. \label{fig:shift_proj_2d}}
	\end{center}
\end{figure}

\subsection{A semi-Lagrangian DG method for Vlasov-Poisson}
\label{sec:vlasov_solver}
In this section, we describe in detail a semi-Lagrangian discontinuous
Galerkin scheme for solving a quasi-1D problem of the form:
\begin{equation}
\label{eqn:quasi1d_adv}
	f_{,t} + a(v) f_{,x} = 0, \quad (t,x,v) \in \reals^{+} \times \reals \times \reals.
\end{equation}
This resulting scheme can then be inserted into Steps 1, 3, and 4
of Algorithm \ref{alg1} for solving the 1+1 dimensional Vlasov-Poisson
system given by \eqref{eqn:VP1d} and \eqref{eqn:VP1dp}. 
We note that in Steps 1 and 4, $a(v)=v$; while
in Step 3 the roles of $x$ and $v$ are reversed and $a(x)=E^{n+\frac{1}{2}}(x)$.

We construct a Cartesian grid over $\Omega = 
[-L, \, L] \times [-V_{\text{max}}, \, V_{\text{max}}]$,
with uniform grid spacing $\Delta x$ and $\Delta v$ in each
coordinate direction. The mesh element ${\mathcal T}_{ij}$ is centered 
at the coordinates
\[
          x_i =  -L  +  \left( i - \frac{1}{2} \right)  \Delta x
          \quad \text{and} \quad
          v_j =  -V_{\text{max}}  +  \left( j - \frac{1}{2} \right) \Delta v.
\]
each element can be mapped to the canonical element via
the simple linear transformation:
\begin{equation}
\label{eqn:canon_vars}
	x = x_i + \xi \, \frac{\Delta x}{2}, \quad 
	v = v_j + \eta \, \frac{\Delta v}{2}.
\end{equation}

Next, we further subdivide each element by introducing
for each horizontal row of elements, $j$, a set of $M$
horizontal lines located at
\begin{equation}
\label{eqn:vsubjk}
       v_{jk} :=  v_j  + \eta_k \, \frac{\Delta v}{2}, \quad
       \text{for} \quad k=1,\ldots,M,
\end{equation}
where $\eta_k$ are the roots of the $M^\text{th}$ degree Legendre
polynomial.  Along each of these lines, we pretend that
we are solving a 1D constant coefficient advection of the form:
\begin{equation}
      f_{,t} + a(v_{jk}) f_{,x} = 0.
\end{equation}
This equation is then essentially solved by the 1D method described 
in the previous section: \S \ref{sec:1dadv}.
This is depicted
in Figure \ref{fig:shift_proj_2d}, where we have chosen $M=3$ for illustration 
purposes\footnote{The parameter $M$ is chosen to coincide with the
spatial  order of accuracy of the scheme. In the numerical examples section we always take $M=5$ to coincide with a fifth-order accurate discretization in
$x$ and $v$.}. 
Finally, a fully 2D solution is reconstructed  in each element,
$\Tm_{ij}$, by summing the $M$ 1D solutions for $k=1,\ldots,M$
with the appropriate  Gaussian quadrature weights.
This algorithm  is summarized in Algorithm \ref{alg2}.

\begin{figure}[!t]
\begin{center}
   \includegraphics[width=120mm]{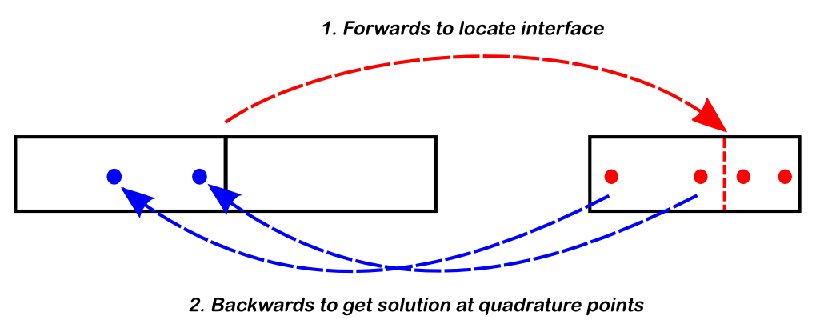}
   \caption{Illustration of the forward and backward nature of the
   proposed semi-Lagrangian scheme. First, the cell edges are
   propagated {\it forward} from their initial time to their final time. 
   Once these locations are known, Gauss-Legendre quadrature
   points are placed between the old cell edges and the new cell
   edges. In order to find solution values at these Gauss-Legendre
   points, we trace {\it backwards} along the characteristics to the
   initial time.\label{fig:forward_backward}}
	\end{center}
\end{figure}

\begin{algorithm}
\caption{The proposed semi-Lagrangian algorithm for solving quasi-1D
	advection equations of the form $f_{,t} + a(v) f_{,x} = 0$.}
\label{alg2}
\begin{algorithmic}
 \STATE 0. Start with the current solution:
 \begin{equation}
   f^h(t^n,x,v) \biggl|_{\Tm_{ij}} := \sum_{\ell=1}^{M(M+1)/2}
   	F^{(\ell)}_{ij} \varphi^{(\ell)}(\xi,\eta),
\end{equation}
where $M$ is the desired order of accuracy in $x$ and $v$. The
canonical variables $(\xi,\eta)$ are linearly related to the variables
$(x,v)$ via \eqref{eqn:canon_vars}.
\STATE 1. In each element, construct $M$ horizontal lines
given by \eqref{eqn:vsubjk}, where $\eta_k$ are the $M$ Gauss-Legendre
quadrature points.  The solution in element $\Tm_{ij}$
along each one of these lines is
\begin{equation}
   \sum_{\ell=1}^{M(M+1)/2}
   	F^{(\ell)}_{ij} \varphi^{(\ell)}(\xi,\eta_k).
\end{equation}
\STATE 2. Advect the cell interfaces forward in time through
1D advection along each $v_{jk}$. After forward advection
along $v_{jk}$,  the $i^{\text{th}}$ cell will contain the {\bf old} interface
$i-\half-I_{jk}$ and this interface will be located a distance
$\Delta x \, \nu_{jk}$ from the {\bf new} interface $i-\frac{1}{2}$,
where
 \begin{align}
 \label{eqn:Ijk_and_nujk}
I_{jk}:= \Bigg\lfloor \frac{a(v_{jk}) \Delta t}{\Delta x} \Bigg\rfloor
\qquad \text{and} \qquad \nu_{jk} := \frac{a(v_{jk}) \Delta t}{\Delta x} - I_{jk}.
\end{align}
\STATE  This is the forwards phase illustrated in
Figure \ref{fig:forward_backward}.
\STATE
\STATE 3. Next we trace characteristics backwards in time 
for each $v_{jk}$. The resulting process yields:
\begin{equation}
\begin{split}
\hspace{-5mm}
S_{ijk}^{(\ell)} &:= \frac{1}{2} \sum_{m=1}^{M(M+1)/2} F_{i-1-I_{jk} \, j}^{(m)}
	\int_{-1}^{-1 + 2 \nu_{jk}} \hspace{-2mm} \varphi^{(m)}(\xi+2-2\nu_{jk},\eta_k)  \, \varphi^{(\ell)}(\xi,\eta_k) \, d\xi  \\
	& \, \, + \frac{1}{2} \sum_{m=1}^{M(M+1)/2} F_{i-I_{jk} \, j}^{(m)}
	\int_{-1+2\nu_{jk}}^{1}  \varphi^{(m)}(\xi-2\nu_{jk},\eta_k)   \, \varphi^{(\ell)}(\xi,\eta_k) \, d\xi,
\end{split}
\end{equation}
\STATE where each of the above integrals are evaluated using 1D Gauss-Legendre
quadrature rules with $M$ points. 
This step in the algorithm is the backwards phase illustrated in
Figure \ref{fig:forward_backward}.
\STATE
\STATE 4. Finally, we update the solution by integrating in the vertical
	direction:
\begin{equation}
   F^{(\ell),\text{new}}_{ij}  =  \sum_{k=1}^{M} \omega_k 
   	\, S_{ijk}^{(\ell)},
\end{equation}
where $\omega_k$ are the usual Gauss-Legendre quadrature weights
for the $M$  Gauss-Legendre quadrature points $\eta_k$.
\end{algorithmic}
\end{algorithm}

In the semi-Lagrangian way of thinking there are generally two
philosophies: 
\begin{enumerate}
\item {\it Forward evolution:} Each ``particle''\footnote{A ``particle'' can be a quadrature point, element interface, etc$\ldots$ depending on the
details of the particular method.} is advected
forward in time and then a solution at the the future time is 
reconstructed (e.g., Particle-in-cell \cite{hockney1988computer, birdsall2004plasma} and Crouseilles et
al. \cite{article:CrReSo09}).
\item {\it Backward evolution:} The solution at a specific location
in the future is determined by tracing backwards along
characteristics (e.g., Restelli et al. \cite{article:ReBoSa06}).
\end{enumerate}
We interpret the semi-Lagrangian DG method described in Algorithm \ref{alg2} 
as a sort of mixed forward and backward method:
\begin{enumerate}
\item {\it Forward evolution phase:} The cell edges are
   propagated {\it forward} from their initial time to their final time.  This
   process determines the set of quadrature points necessary for the projection step.
   This forward evolution step is needed in order to attain mass
   conservation.
\item {\it Backward evolution phase:}
   Once the old cell edge locations are known at the new time,
   Gauss-Legendre quadrature
   points are placed between the old cell edges and the new cell
   edges. In order to find solution values at these Gauss-Legendre
   points, we trace {\it backwards} along the characteristics to the
   initial time.
\end{enumerate}
This process is illustrated in Figure \ref{fig:forward_backward}.

\subsection{Poisson solver}
\label{sec:poisson}
We describe in this section how to efficiently solve the
Poisson equation in Step 2 of the operator splitting approach
shown in Algorithm \ref{alg1}.
Consider first the 1D Poisson equation on $x \in [a,b]$
with mixed boundary conditions:
\begin{equation}
	\label{eqn:poisson_1d}
   \phi_{,xx} = \rho(x)-\rho_0, \quad \phi_{,x}(a) = \gamma, \quad \phi(b) = \beta.
\end{equation}
We apply to the Poisson equation the so-called local discontinuous Galerkin method (LDG) (see \cite{article:ABCM02,book:HesWar07} for two reviews of various approaches for solving Poisson equations via the DG method),
and rewrite it as a system of two equations:
\begin{align}
\label{eqn:poisson1}	 
   E_{,x} &= \rho(x)-\rho_0, \\
\label{eqn:poisson2}
	\phi_{,x} &= E(x). 
\end{align}
We expand $\phi(x)$, $E(x)$, and $\rho(x)$ on each element as follows:
\begin{equation}
   \Bigl\{ \phi^{h}(x), \,
   E^{h}(x), \, \rho^{h}(x) - \rho_0 \Bigr\} \biggl|_{{\mathcal T}_{i}}
   = \sum_{k=1}^{M} \left\{ \Phi^{(k)}_i, \, {\mathbb E}^{(k)}_i, \, {\mathbb P}^{(k)}_i \right\}
    \, \varphi^{(k)}_{\text{1D}}(\xi),
\end{equation}
where  $M$ is the desired order of accuracy.

We multiply (\ref{eqn:poisson1}) and (\ref{eqn:poisson2})
each by $\varphi^{(\ell)}_{\text{1D}}(\xi)$  and integrate from $\xi=-1$ to $\xi=1$:
\begin{align}	
  \label{eqn:weak_poisson1}
    \frac{1}{\Delta x} \left[  \varphi^{(\ell)}_{\text{1D}} \, E^{h} \right]_{-1}^{1}
     - \frac{1}{\Delta x} \int_{-1}^{1}   \varphi^{(\ell)}_{{\text{1D}},\xi} \, 
  E^{h} \, d\xi  &= {\mathbb P}^{(\ell)}_i , \\
  \label{eqn:weak_poisson2}
   \frac{1}{\Delta x} \left[  \varphi^{(\ell)}_{\text{1D}} \, \phi^{h} \right]_{-1}^{1}
  - \frac{1}{\Delta x} \int_{-1}^1   \varphi^{(\ell)}_{{\text{1D}} ,\xi} \, 
  \phi^{h} \, d\xi &= {\mathbb E}^{(\ell)}_i.
\end{align}
Next, we apply the following {\it one-sided} rules in order to evaluate
$\phi^{h}$ and $E^{h}$ at the grid interfaces:
\begin{align}
	\phi^{h} (-1)  & := 
	\sum_{k=1}^{M} \, \varphi^{(k)}_{\text{1D}} (-1) \, \Phi^{(k)}_i
	= \sum_{k=1}^{M} (-1)^{k+1} \, \sqrt{2k-1} \, \, \Phi^{(k)}_i, \\
	\phi^{h} (1)  & := 
	\sum_{k=1}^{M} \, \varphi^{(k)}_{\text{1D}} (-1) \, \Phi^{(k)}_{i+1}
	= \sum_{k=1}^{M} (-1)^{k+1} \, \sqrt{2k-1} \, \, \Phi^{(k)}_{i+1}, \\
	E^{h} (-1)  & := 
	\sum_{k=1}^{M} \, \varphi^{(k)}_{\text{1D}} (1)  \, {\mathbb E}^{(k)}_{i-1}
	= \sum_{k=1}^{M} \sqrt{2k-1} \, \, {\mathbb E}^{(k)}_{i-1}, \\
	E^{h} (1)  & := 
	\sum_{k=1}^{M} \, \varphi^{(k)}_{\text{1D}} (1)  \, {\mathbb E}^{(k)}_{i}
	= \sum_{k=1}^{M} \sqrt{2k-1} \, \, {\mathbb E}^{(k)}_{i}.
\end{align}
Using these definitions, (\ref{eqn:weak_poisson1})
and (\ref{eqn:weak_poisson2}) can be rewritten as follows:
\begin{align}
\label{eqn:discrete_poisson1}
   \sum_{k=1}^{M}  \sqrt{2k-1} \, \sqrt{2\ell-1} \,
    \left( {\mathbb E}^{(k)}_{i} + (-1)^{\ell} \, {\mathbb E}^{(k)}_{i-1} \right) 
    -  S_{\ell k} \, {\mathbb E}^{(k)}_i 
     &=  \Delta x \, {\mathbb P}_i^{(\ell)}, \\
     \label{eqn:discrete_poisson2}
  \sum_{k=1}^{M} (-1)^{k+1} \, \sqrt{2k-1} \, \sqrt{2\ell-1} \,
    \left( \Phi^{(k)}_{i+1} + (-1)^{\ell} \, \Phi^{(k)}_{i} \right) 
    - S_{\ell k} \, \Phi^{(k)}_i
      &=  \Delta x \, {\mathbb E}_i^{(\ell)},
\end{align}
where $S$ is an $M \times M$ matrix with entries given by
\begin{equation}
   S_{\ell k} =  \int_{-1}^{1} 
     \varphi^{(\ell)}_{{\text{1D}}, \xi} \, \varphi^{(k)}_{\text{1D}}  \, d\xi.
\end{equation}
Note that the boundary conditions in (\ref{eqn:poisson_1d})
imply that
\begin{gather}
   {\mathbb E}^{h}(a) = \gamma \quad \Longrightarrow \quad
   \sum_{k=1}^{M} \sqrt{2k-1} \, {\mathbb E}_{0}^{(k)} = \gamma, \\
  \phi^{h}(b) = \beta \quad \Longrightarrow \quad
   \sum_{k=1}^{M} (-1)^{k+1} \sqrt{2k-1} \, \Phi_{m_x+1}^{(k)} = \beta,
\end{gather}
where $m_x$ is the number of grid elements.

Putting everything together,
 (\ref{eqn:discrete_poisson1}) and (\ref{eqn:discrete_poisson2}) can be written in
 matrix form:
\begin{align}
\label{eqn:poisson_dg1}
\frac{1}{\Delta x} \begin{bmatrix}
A &    &           &             &   \\ 
B & A &           &             &     \\
   & B & A        &             &     \\
   &    & \ddots & \ddots  &     \\
   &    &         	 & B          & A 
   \end{bmatrix}
   \begin{bmatrix}
   	\vec{\mathbb E}_1 \\ \vec{\mathbb E}_2 \\ \vec{\mathbb E}_3 \\ \vdots \\ \vec{\mathbb E}_{m_x}
   \end{bmatrix} &= 
   \begin{bmatrix}
   	 \vec{\mathbb P}_1 - 
	(-1)^{\ell} \, \sqrt{2\ell - 1} \, \gamma \left( \Delta x \right)^{-1}
	\\  \vec{\mathbb P}_2 \\  \vec{\mathbb P}_3
	 \\ \vdots \\ \, \vec{\mathbb P}_{m_x}
   \end{bmatrix}, \\
   \label{eqn:poisson_dg2}
 \frac{1}{\Delta x}   \begin{bmatrix}
C & D &          &             &   \\ 
   & C & D       &             &     \\
   &     & C       &  \ddots &     \\
   &    &            & \ddots & D       \\
   &    &            &        & C 
   \end{bmatrix}
   \begin{bmatrix}
   	\vec{\Phi}_1 \\ \vec{\Phi}_2 \\ \vec{\Phi}_3 \\ \vdots \\ \vec{\Phi}_{m_x}
   \end{bmatrix} &= 
   \begin{bmatrix}
   	 \vec{\mathbb E}_1 \\  \vec{\mathbb E}_2 \\ \vec{\mathbb E}_3
	 \\ \vdots \\  \vec{\mathbb E}_{m_x} - \sqrt{2\ell - 1} \, \beta \,  \left( \Delta x \right)^{-1}
   \end{bmatrix},
\end{align}
where, for example,
\begin{gather}
	\vec{\mathbb E}_i = \left( {\mathbb E}^{(1)}_i, \, \ldots , \, {\mathbb E}^{(M)}_i \right)^{T}, 
\end{gather}
and $A$, $B$, $C$, and $D$ are $M \times M$ matrices
with entries given by:
\begin{align}
  A_{\ell k} &= \sqrt{2k-1} \sqrt{2\ell-1} - S_{\ell k}, \\
  B_{\ell k} &= (-1)^{\ell} \sqrt{2k-1} \sqrt{2\ell-1}, \\
  C_{\ell k} &= (-1)^{k+\ell+1} \sqrt{2k-1} \sqrt{2\ell-1} - S_{\ell k}, \\
  D_{\ell k} &= (-1)^{k+1} \sqrt{2k-1} \sqrt{2\ell-1}.
\end{align}
The advantage of this formulation is that equations \eqref{eqn:poisson_dg1}
and  \eqref{eqn:poisson_dg2} are already in lower and upper triangular forms,
respectively, and therefore can be easily be solved. The matrices $A$ and $C$ 
can be easily inverted once at the beginning of the calculation.

\subsubsection{Dirichlet boundary conditions}
The above method can easily be adapted to handle Dirichlet boundary conditions:
\begin{equation}
	\label{eqn:poisson_dirichlet}
   \phi(a) = \alpha, \quad \phi(b) = \beta,
\end{equation}
by noting that these boundary conditions 
are equivalent to the mixed BCs in  (\ref{eqn:poisson_1d}) if
we carefully choose the parameter $\gamma$ in (\ref{eqn:poisson_1d}). It can be shown that the correct
choice for $\gamma$ is given by
\begin{equation}
\begin{split}
    \gamma &= \frac{\beta - \alpha}{b-a} + 
    			\frac{1}{b-a} \int_{a}^{b} \left( s - b \right) \Bigl( \rho(s) - \rho_0 \Bigr) \, ds \\
			&= \frac{\beta - \alpha}{b-a} + 
			   			\frac{\Delta x}{b-a} \sum_{i=1}^{m_x} \left(  x_i - b \right) {\mathbb P}^{(1)}_i
			+ \frac{\Delta x^2}{2 \sqrt{3} \left(b - a\right)} \sum_{i=1}^{m_x} {\mathbb P}^{(2)}_i,
\end{split}
\end{equation}
where $m_x$ is the number of grid elements.

\subsubsection{Periodic boundary conditions}
Periodic boundary conditions can also be readily handled:
\begin{equation}
	\label{eqn:poisson_periodic}
   \phi(a) = \phi(b), \quad E(a) = E(b),
\end{equation}
by again noting that we need to carefully choose
 $\beta$ and $\gamma$ in (\ref{eqn:poisson_1d}). It can be shown that the correct
choice for $\gamma$ is given by
\begin{equation}
   \label{eqn:gamma_periodic}
    \gamma =   \frac{1}{b-a} \int_{a}^{b} s \Bigl( \rho(s) - \rho_0 \Bigr) \, ds
    = \frac{\Delta x}{b-a} \sum_{i=1}^{m_x} x_i {\mathbb P}^{(1)}_i
			+ \frac{\Delta x^2}{2 \sqrt{3} \left(b-a \right)} \sum_{i=1}^{m_x} {\mathbb P}^{(2)}_i,
\end{equation}
and $\beta$ is arbitrary. Without loss of generality we simply take $\beta = 0$.
By solving (\ref{eqn:poisson_1d}) with $\gamma$ given by
(\ref{eqn:gamma_periodic}) and with $\beta=0$, we obtain a solution $\phi(x)$ with the property that
\begin{equation}
   \phi(a) = \phi(b) = 0.
\end{equation}

\subsection{High-order split semi-Lagrangian method}
\label{sec:higherordersplit}
Now that the basic pieces are in place (i.e., a semi-Lagrangian
solver for each split-piece \S\ref{sec:vlasov_solver} and the Poisson-solver \S\ref{sec:poisson}),
we are ready to introduce a fully fourth-order accurate method
for the Vlasov-Poisson system. We describe in this section some important implementation
 details needed to achieve this. The resulting method is
 summarized in Algorithm \ref{alg3}.

\subsubsection{Fourth-order operator splitting}
Consider a time-dependent problem where the right-hand
side is written as the sum of two differential operators ${\mathcal A}$ and ${\mathcal B}$:
\[
	q_{,t} = {\mathcal A}(q) + {\mathcal B}(q).
\]
A fourth-order accurate operator splitting technique  for such systems  was
developed by Forest and Ruth \cite{article:FoRu90} 
and by Yoshida \cite{article:Yoshida90,article:Yoshida93}.
If we define the following two constants:
\begin{align}
   \gamma_1 &=  \frac{1}{2 - 2^{1/3}} \approx 1.351207191959658, \\
   \gamma_2 &= -\frac{2^{1/3}}{2 - 2^{1/3}} \approx -1.702414383919315,
\end{align}
then the fourth-order splitting approach of 
\cite{article:FoRu90,article:Yoshida90,article:Yoshida93} 
can be written as a composition of the following seven stages:
\begin{align*}
\text{\bf Stage 1:} & \quad \frac{\gamma_1  \Delta t }{2} \approx 0.6756 \, \Delta t 
\quad \text{step on} \quad 
q_{,t} =  {\mathcal A}\left(  q \right), \\
\text{\bf Stage 2:} & \quad   {\gamma_1  \Delta t }  \approx
		 1.3512 \, \Delta t \quad \text{step on} \quad
		 q_{,t} = {\mathcal B}(q), \\
\text{\bf Stage 3:} & \quad \frac{(\gamma_1 + \gamma_2) \Delta t }{2} \approx
		 -0.1756 \, \Delta t \quad \text{step on} \quad
		 q_{,t} = {\mathcal A}(q),  \\
\text{\bf Stage 4:} & \quad  \gamma_2 \Delta t \approx
		-1.7024 \, \Delta t \quad \text{step on} \quad q_{,t} = {\mathcal B}(q), \\
\text{\bf Stage 5:} &  \quad \frac{(\gamma_1 + \gamma_2) \Delta t }{2} \approx
		 -0.1756 \, \Delta t \quad \text{step on} \quad
		 q_{,t} = {\mathcal A}(q),  \\
\text{\bf Stage 6:} & \quad   {\gamma_1  \Delta t }  \approx
		 1.3512 \, \Delta t \quad \text{step on} \quad
		 q_{,t} = {\mathcal B}(q), \\
\text{\bf Stage 7:} & \quad \frac{\gamma_1  \Delta t }{2} \approx 0.6756 \, \Delta t 
\quad \text{step on} \quad 
q_{,t} =  {\mathcal A}\left(  q \right).
\end{align*}
We note that this splitting approach requires some steps larger than $\Delta t$: 
{\bf Step 2} and {\bf Step 6}; as well as backward steps:
 {\bf Step 3},  {\bf Step 4}, and  {\bf Step 5}.

 \subsubsection{Application to Vlasov-Poisson}
 \label{sec:high_order_semiL_alg}
One difficulty with raising the temporal order of accuracy from
two to four is the time-dependence
of the electric field. In other words, the Cheng and Knorr
\cite{article:ChKn76} method does not completely reduce the
Vlasov-Poisson to two constant coefficient problems, since
the electric field remains time-dependent. In the case of 
Strang splitting, it turned out that one could easily generate
a second order accurate representation of the electric field
at the half time step, $t^n + \frac{1}{2} \Delta t$, as required in Step 3 of
Algorithm \ref{alg1}, simply by carrying out Steps 1 and 2 of 
Algorithm \ref{alg1}. Additional attention must be paid
in order to obtain temporally fourth-order accurate
representations of the electric field.
  
In order to avoid having to use the electric field at different
points in time (i.e., a multi-step method), we construct
the fourth-order Taylor polynomial centered at $t=t^n$:
\begin{equation}
\bar{\E}(t,{\bf x}) := \E^n +  \left( t - t^n \right) \E^n_{,t}
+ \frac{1}{2} \left( t - t^n \right)^2 \E^n_{,tt}
+ \frac{1}{6} \left( t - t^n \right)^3 \E^n_{,ttt}.
\end{equation}
The electric field value $\E^n$ is computed from
the Poisson equation:
\begin{equation}
  \nabla \cdot \E^n = \nabla^2 \phi^n = \rho^n - \rho_0.
\end{equation}
The first time derivative of the electric field is
proportional to the momentum:
\begin{align}
\label{eqn:Efield_time_derivative}
    \nabla \cdot \E_{,t} = \rho_{,t} = -\nabla \cdot \left(
   	\rho {\bf u} \right) \quad \Longrightarrow \quad
	  \E^n_{,t} = - \left( \rho {\bf u} \right)^n,
\end{align}
and is thus readily computable. 
In order to compute the remaining time derivatives,
we write down the evolution equations for the
first few moments of the Vlasov-Poisson equation:
\begin{align}
	\rho_{,t} + \nabla \cdot \left( \rho {\bf u} \right) &= 0, \\
	\left( \rho {\bf u} \right)_{,t} + \nabla \cdot {\mathbb E} &= \rho \E,  \\
	{\mathbb E}_{,t} + \nabla \cdot {\mathbb F} &= 
	\rho \left( {\bf u} \E + \E {\bf u} \right),
\end{align}
where
\begin{equation*}
  \rho := \int_{\bf v} f \, d{\bf v}, \quad
  \rho {\bf u}  := \int_{\bf v} {\bf v}  f \, d{\bf v}, \quad
  {\mathbb E}  := \int_{\bf v} {\bf v} {\bf v}  f \, d{\bf v}, \quad \text{and} \quad
  {\mathbb F}  := \int_{\bf v} {\bf v} {\bf v} {\bf v}  f \, d{\bf v}.
\end{equation*}
Using equation \eqref{eqn:Efield_time_derivative} and
the above moment evolution equations, we can compute
the second and third time derivatives of the electric field entirely in
terms of spatial derivatives:
\begin{align}
\label{eqn:Ett}
\E_{,tt} &= - \left( \rho {\bf u} \right)_{,t} = 
\nabla \cdot {\mathbb E} - \rho \, \E, \\
\begin{split}
\E_{,ttt} &=  
\nabla \cdot {\mathbb E}_{,t} - \rho_{,t} \, \E - \rho \, \E_{,t} \\
\label{eqn:Ettt}
&= \nabla \cdot \left( \rho
{\bf u} \E +  \rho  \E {\bf u}  \right) - \nabla \cdot \nabla \cdot {\mathbb F} 
+ \E \, \nabla \cdot \left( \rho {\bf u} \right)
+ \rho^2 {\bf u}.
\end{split}
\end{align}

It is clear from these expressions that 
in order to compute $\E_{,tt}$ and $\E_{,ttt}$, we
need to be able to compute first and second derivatives in
space. One approach for doing this in the discontinuous 
Galerkin framework is to multiply by a test function and then
integrate-by-parts. However, this approach will in general lead to a loss
of accuracy. Instead, the approach taken in this work is to apply
central finite differences that work directly on the Legendre coefficients of
the function that needs to be differentiated. 

Consider the $L_2$-projection of the function $f(x)$, where $f:\reals \rightarrow \reals$,
onto the space of piecewise polynomials of degree four on 
a uniform mesh of elements, ${\mathcal T}_i$, that each have
width $\Delta x$:
\begin{equation}
	f^h \Bigl|_{\Tm_i} = \sum_{\ell=1}^5  F^{(\ell)}_i \, 
	\varphi^{(\ell)}_{\text{1D}}(\xi).
\end{equation}
Therefore, $f^h$ represents the finite dimensional approximation of
$f(x)$. We can approximate the first and second derivatives
of $f(x)$ to ${\mathcal O}\left(\Delta x^5\right)$ accuracy
by computing appropriate central finite differences
of the Legendre cofficients $F^{(\ell)}$. If we let
$D_x f^h$ and $D_{xx} f^h$ represent the
finite dimensional approximations of $f'(x)$ and
$f''(x)$, respectively, then the central
finite difference formulas  on the Legendre coefficients are
\begin{align}
\label{eqn:finite_diff_1}
 \begin{bmatrix}
  	D_{x} F^{(1)}_i \\ D_{x} F^{(2)}_i \\ D_{x} F^{(3)}_i \\
	   D_{x} F^{(4)}_i \\ D_{x} F^{(5)}_i
  \end{bmatrix} &= 
  \frac{1}{2 \Delta x} 
    \begin{bmatrix}
  	\Delta_1 F^{(1)}_i - 2 \sqrt{5} \, \Delta_1 F^{(3)}_i + 78 \, \Delta_1 F^{(5)}_i \\ 
	 \Delta_1 F^{(2)}_i - \frac{10}{3} \sqrt{3} \sqrt{7} \, \Delta_1 F^{(4)}_i \\
	 	 	\Delta_1 F^{(3)}_i -14 \sqrt{5} \,  \Delta_1  F^{(5)}_i \\
	 \Delta_1 F^{(4)}_i  \\ 
	 	 	\Delta_1 F^{(5)}_i  
  \end{bmatrix}, \\
  \label{eqn:finite_diff_2}
  \begin{bmatrix}
  	D_{xx} F^{(1)}_i \\ D_{xx} F^{(2)}_i \\ D_{xx} F^{(3)}_i \\
	   D_{xx} F^{(4)}_i \\ D_{xx} F^{(5)}_i
  \end{bmatrix} &= 
  \frac{1}{\Delta x^2}
\begin{bmatrix}
  	\Delta_2 F^{(1)}_i -  \sqrt{5} \, \Delta_2 F^{(3)}_i + 11 \, \Delta_2 F^{(5)}_i \\ 
	 \Delta_2 F^{(2)}_i - \frac{5}{3} \sqrt{3} \sqrt{7} \, \Delta_2 F^{(4)}_i \\
	 	\Delta_2 F^{(3)}_i -7 \sqrt{5} \,  \Delta_2  F^{(5)}_i \\
	 \Delta_2 F^{(4)}_i  \\ 
	 	\Delta_2 F^{(5)}_i
  \end{bmatrix},
\end{align}
where 
\begin{align}
\Delta_1 F^{(k)}_i &:=  F^{(k)}_{i+1} - F^{(k)}_{i-1}, \\
\Delta_2 F^{(k)}_i &:=  F^{(k)}_{i+1} - 2 F^{(k)}_i + F^{(k)}_{i-1}.
\end{align}
To the best of our knowledge, this is the first
time such formulas have been written down
in the context of discontinuous Galerkin methods.
In Table \ref{table:fd_verify} we verify the order
of accuracy by computing the first and second derivatives
of $f(x) = e^{\sin(2 \pi x)}$. The errors
in this table are computed using the relative $L_2$ errors defined by equation
\eqref{eqn:L2error_1D} with $M=5$ and varying $\Delta x$.
See \ref{sec:L2error_1D} for more details.

\begin{table}
\begin{center}
\begin{Large}
  \begin{tabular}{|c||c|c||c|c|}
  \hline
  {\normalsize {\bf Mesh}} & {\normalsize \bf $f'(x)$ error} 
  & {\normalsize\text{$\log_2${\bf (Ratio)}}} & {\normalsize \bf $f''(x)$ error} 
  & {\normalsize\text{$\log_2${\bf (Ratio)}}} \\
    \hline \hline
{\normalsize $25$}   &   {\normalsize $1.747\times 10^{-4}$}   & {\normalsize --} &  {\normalsize $8.292\times 10^{-5}$} & {\normalsize --}  \\
\hline
{\normalsize $50$}   &   {\normalsize $5.543\times 10^{-6}$}   & {\normalsize 4.98} &   {\normalsize $2.672\times 10^{-6}$}  & {\normalsize 4.96}  \\
\hline
{\normalsize $100$}   &   {\normalsize $1.738\times 10^{-7}$}   &  {\normalsize 5.00} & {\normalsize $8.413\times 10^{-8}$}  & {\normalsize 4.99} \\
\hline
{\normalsize $200$}   &   {\normalsize $5.437\times 10^{-9}$}   &  {\normalsize 5.00} & {\normalsize $2.634\times 10^{-9}$}  & {\normalsize 5.00}  \\
\hline
{\normalsize $400$}   &   {\normalsize $1.699\times 10^{-10}$}   &  {\normalsize 5.00} & {\normalsize $8.364\times 10^{-11}$}  & {\normalsize 4.98}  \\
 \hline
  \end{tabular}
  \end{Large}
  \caption{Relative $L_2$-norm errors for computing $f'(x)$ and $f''(x)$ 
  using the Legendre coefficient finite difference formulas 
  \eqref{eqn:finite_diff_1} and \eqref{eqn:finite_diff_2}, respectively.
  The example shown here is for
  	$f(x)=e^{\sin(2\pi x)}$ on a uniform mesh on $0 \le x \le 1$.
   Periodic boundary conditions are imposed to compute the derivative
   in the first and last elements: $F^{(k)}_0 = F^{(k)}_{M}$ and $F^{(k)}_{M+1} = F^{(k)}_{1}$
   for each $k=1,2,3,4,5$.
   \label{table:fd_verify}}
 \end{center}
\end{table}

Once we have constructed an approximation to the 
time-dependent electric field, we are faced with an
advection equation with time-dependent coefficients:
\begin{equation}
    f_{,t}  + \bar{\E}(t,{\bf x}) \cdot f_{,\v} = 0.
 \end{equation}
This equation can be readily solved to high-order via the
method of characteristics. The key step in this approach
is the evolution of the coordinates $\v$ as function of time:
 \begin{gather}
    \frac{d\v}{dt} = \bar{\E}(t,{\bf x}) \quad \Longrightarrow \quad
    	\v(t^n + \Delta t) = \v(t^n) + \int_{t^n}^{t^n + \Delta t} \bar{\E}(t,{\bf x}) \, dt, \\
	\label{eqn:time_dep_move}
  \v(t^n + \Delta t) = \v(t^n) + \Delta t \, \E^n
  + \frac{\Delta t^2}{2} \, \E_{,t}^n
  +  \frac{\Delta t^3}{6} \, \E_{,tt}^n
  +  \frac{\Delta t^4}{24} \, \E_{,ttt}^n.
 \end{gather}
 In other words, the semi-Lagrangian DG method as outlined
 in \S \ref{sec:vlasov_solver}, remains largely unaltered by the
 fact that the electric is time dependent. The only difference is that
 interfaces and quadrature points are transported by equation 
 \eqref{eqn:time_dep_move} instead of the simpler version of
 this equation when $\bar{\E}$ is constant in time.
 
We note than an important advantage of
this approach is that, just as with Strang splitting,  it requires
only a single Poisson solver per time step.
 Finally, we summarize the complete fourth-order splitting method
 for the Vlasov-Poisson system in Algorithm \ref{alg3}. 
 
  \begin{algorithm}
\caption{Fourth-order operator split algorithm.}
\label{alg3}
\begin{algorithmic}
 \STATE 1.   Solve \quad $\nabla^2 \phi = \rho^{n} -  \rho_0$ \, \text{and} \,
compute \quad ${\bf E}^{n} = \nabla \phi$.
\STATE
\STATE 2. Compute 
\begin{align*} 
{\bf E}^{n}_{,t}  &= - \left( \rho {\bf u} \right)^n, \\
{\bf E}^{n}_{,tt}  &= \nabla \cdot {\mathbb E}^n - \left( \rho \E \right)^n, \\
{\bf E}^{n}_{,ttt}  &= \nabla \cdot \left( \rho
{\bf u} \E +  \rho  \E {\bf u}  \right)^n - \nabla \cdot \nabla \cdot {\mathbb F}^n
+ \E^n \, \nabla \cdot \left( \rho {\bf u} \right)^n
+ \left( \rho^2 {\bf u} \right)^n,
\end{align*}
\STATE \hspace{4mm} where the spatial derivatives are computed via
 \eqref{eqn:finite_diff_1} and \eqref{eqn:finite_diff_2}.
\STATE  \hspace{4mm} Then construct 
\[  
     \bar{\E}(t,{\bf x}) := \E^n+  (t-t^n) \, \E^n_{,t}
	+  \frac{1}{2}  (t-t^n)^2 \, \E^n_{,tt} + \frac{1}{6}  (t-t^n)^3 \, \E^n_{,ttt}.
\]
\STATE
  \STATE 3. $\frac{\gamma_1}{2} \Delta t$\quad step on \quad $f_{,t} + {\bf v} \cdot f_{,{\bf x}} = 0$.
 \STATE
   \STATE 4. $\gamma_1 \Delta t$\quad step on \quad $f_{,t} + \bar{\bf E}\left(t, {\bf x} \right) \cdot f_{,{\bf v}} = 0$, \quad
   $t\in t^n + \Bigl[0, \, \gamma_1 \Delta t \Bigr]$.
   \STATE
   \STATE 5. $\frac{(\gamma_1 + \gamma_2)}{2}  \Delta t$  \quad step on \quad $f_{,t} + {\bf v} \cdot f_{,{\bf x}} = 0$.
    \STATE
   \STATE 6. $\gamma_2 \Delta t$ \quad step on \quad $f_{,t} + \bar{\bf E}\left( t, {\bf x} \right) \cdot f_{,{\bf v}} = 0$, \quad $t\in t^n + \Bigl[ \gamma_1 \Delta t, \, (\gamma_1+\gamma_2)
   	 \Delta t \Bigr]$.
   \STATE
   \STATE 7. $\frac{(\gamma_1 + \gamma_2)}{2}  \Delta t$  \quad step on \quad $f_{,t} + {\bf v} \cdot f_{,{\bf x}} = 0$.
 \STATE
   \STATE 8. $\gamma_1 \Delta t$\quad step on \quad $f_{,t} + \bar{\bf E}\left(t, {\bf x}
   	\right) \cdot f_{,{\bf v}} = 0$, \quad $t\in t^n + \Bigl[ (\gamma_1 + \gamma_2) \Delta t, \, 
	 (2\gamma_1 + \gamma_2) \Delta t \Bigr]$.
   \STATE
   \STATE 9. $\frac{\gamma_1}{2} \Delta t$\quad step on \quad $f_{,t} + {\bf v} \cdot f_{,{\bf x}} = 0$.
\end{algorithmic}
\end{algorithm}

\subsection{Mass conservation and positivity in the mean}
\label{sec:cons_mass}
The description of the method so far has yielded
a fourth-order accurate in time and fifth-order in
space semi-Lagrangian method for the Vlasov-Poisson
equations. It still remains to show that the method is
mass conservative and that it is positivity-preserving. 
We prove mass conservation in the first theorem below.
This is followed by a proof that each step in the 
operator split semi-Lagrangian produces a solution that it
is positive in the mean\footnote{We show in the next subsection
how to turn positivity in the mean into global positivity.}.

\begin{thm}[Conservation] Each step of the semi-Lagrangian method described
above is mass conservative.
\end{thm}

\begin{proof}
It suffices to show that the semi-Lagrangian scheme is conservative on the 
quasi-1D problem given by \eqref{eqn:quasi1d_adv}
with periodic boundary conditions in $x$.
Using the notation of Algorithm \ref{alg2},
the update for the mean value in cell $\Tm_{ij}$ can be written in the form:
\[
\begin{split}
F^{(1),\text{new}}_{ij}  =& \, \, \sum_{k=1}^M \omega_k \, S^{(1)}_{ijk} \\
= & \, \, \frac{1}{2} \sum_{k=1}^M \sum_{m=1}^{M(M+1)/2} \omega_k \,
 	F_{i-1-I_{jk} \, j}^{(m)}
	\int_{-1}^{-1 + 2 \nu_{jk}} \varphi^{(m)}(\xi+2-2\nu_{jk},\eta_k)  \,
	 d\xi  \\
+ & \, \, \frac{1}{2} \sum_{k=1}^M \sum_{m=1}^{M(M+1)/2} \omega_k \,
	F_{i-I_{jk} \, j}^{(m)}
	\int_{-1+2\nu_{jk}}^{1}  \varphi^{(m)}(\xi-2\nu_{jk},\eta_k)   \,
	 d\xi.
\end{split}
\]
The total integral of $f^h(t,x,v)$ over the entire computational domain 
can then be written as
\begin{align*}
\begin{split}
	\sum_{i,j} F^{(1),\text{new}}_{ij} =&  \, \,  \frac{1}{2}
	\sum_{i,j,k,m} \omega_k \, F_{i-1-I_{jk} \, j}^{(m)}
	\int_{-1}^{-1 + 2 \nu_{jk}} \varphi^{(m)}(\xi+2-2\nu_{jk},\eta_k)  \,
	 d\xi  \\
+ & \, \, \frac{1}{2} \sum_{i,j,k,m} \omega_k \, F_{i-I_{jk} \, j}^{(m)}
	\int_{-1+2\nu_{jk}}^{1}  \varphi^{(m)}(\xi-2\nu_{jk},\eta_k)   \,
	 d\xi.
\end{split}
\end{align*}
We make the following change of variables in the integrals above:
\[
	s=\xi+2-2\nu_{jk} \qquad \text{and} \qquad
	s=\xi-2\nu_{jk},
\]
respectively, which yields:
\begin{align*}
\begin{split}
	\sum_{i,j} F^{(1),\text{new}}_{ij} =&  \, \,  \frac{1}{2}
	\sum_{i,j,k,m} \omega_k \, F_{i-1-I_{jk} \, j}^{(m)}
	\int_{1-2\nu_{jk}}^{1} \varphi^{(m)}(s,\eta_k)  \,
	 ds  \\
+ & \, \, \frac{1}{2} \sum_{i,j,k,m} \omega_k \, F_{i-I_{jk} \, j}^{(m)}
	\int_{-1}^{1-2\nu_{jk}}  \varphi^{(m)}(s,\eta_k)   \,
	 ds.
\end{split}
\end{align*}
Since we are summing over all $i$, we shift the first index of $F$
without changing the total sum; this step allows us to combine
the two integrals into one:
\begin{align*}
	\sum_{i,j} F^{(1),\text{new}}_{ij} =&  \, \,  \frac{1}{2}
	\sum_{i,j,k,m} \omega_k \, F_{ij}^{(m)}
	\int_{-1}^{1} \varphi^{(m)}(s,\eta_k)  \,
	 ds.
\end{align*}
Since $\varphi^{(m)}(s,\eta)$ is polynomial of degree at most $M-1$ in 
$\eta$, the Gaussian quadrature represented by the sum
over $k$ is exact:
\begin{align*}
\sum_{i,j} F^{(1),\text{new}}_{ij} =&  \, \,  
	\sum_{i,j} \sum_m F_{ij}^{(m)} \left[ \frac{1}{4}
	\int_{-1}^{1}  \int_{-1}^{1}  \varphi^{(m)}(s,\eta) \, ds \, d\eta \right] \\
	=& \, \, \sum_{i,j} \sum_m F_{ij}^{(m)} \delta_{1m} = \sum_{i,j} F_{ij}^{(1)},
\end{align*}
where $\delta_{1m}$ is the Kronecker delta.
\end{proof}

\begin{thm}[Positivity in the mean] Let $M$ denote the
\label{thm:positivity}
spatial order of accuracy and let
\begin{equation}
\label{eqn:Kceil}
	K := \left\lceil \frac{M}{2} \right\rceil,
\end{equation}
where $\lceil \cdot \rceil$ denotes the {\it ceiling} operation\footnote{this function
takes a real input and rounds up to the smallest integer that is larger
than or equal to the input.}.
Let $f^h(t^n, x, v)$ be a function defined on the broken finite element space 
\eqref{eqn:broken_space} with $q=M-1$, and let
\begin{equation}
\tilde{f}^h_{ij}(t^n, \xi, \eta) := f^h(t^n, x, v) \Bigl|_{\Tm_{ij}},
\end{equation}
where $(\xi,\eta) \in [-1,1] \times [-1,1]$ are the variables on the canonical
element.
Assume that $\tilde{f}^h_{ij}(t^n, \xi, \eta)$ is
non-negative at all of the following $2MK$ points:
\begin{align}
\label{eqn:pos_pts_1}
\left(\xi, \eta \right) = \left( \xi^{L}_{\ell j k}, \eta_k \right), \qquad \text{where} \quad 
	 \xi^{L}_{\ell j k} &:=  \nu_{jk} \left(1-s_{\ell} \right)+s_{\ell}, \\
\label{eqn:pos_pts_2}
\left(\xi, \eta \right) = \left( \xi^{R}_{\ell j k}, \eta_k \right), \qquad \text{where} \quad 
\xi^{R}_{\ell j k} &:=  \nu_{jk} \left(1+s_{\ell} \right)-1, 
\end{align}
$\forall k=1,\ldots,M$ and $\forall \, \ell=1,\ldots,K$.
In the above expression,  $\nu_{jk}$ is given by \eqref{eqn:Ijk_and_nujk},
$s_{\ell}$ denotes the $\ell^\text{th}$ quadrature point in the standard
1D Gauss-Legendre rule with $K$ points, and $\eta_k$ 
denotes the $k^\text{th}$ quadrature point in the standard
1D Gauss-Legendre rule with $M$ points.

If one time-step in one coordinate direction is taken using the semi-Lagrangian scheme as described in Algorithm \ref{alg2}
with  $f^h(t^n, x, v)$ as the initial condition, then the approximate
solution at the end of this time-step
will have a non-negative average in every element (independent of the time step $\Delta t$):
\[
	F^{(1),\text{new}}_{ij} \ge 0, \quad \forall \, \Tm_{ij} \in \Omega^h.
\]
\end{thm}

\begin{proof}
Using the notation of Algorithm \ref{alg2},
the update for the mean-value in element $\Tm_{ij}$ can be written as
\[
\begin{split}
F^{(1),\text{new}}_{ij}  
= & \, \, \frac{1}{2} \sum_{k=1}^M  \omega_k \,
 	\int_{-1}^{-1 + 2 \nu_{jk}} \left\{ \sum_{m=1}^{M(M+1)/2} F_{i-1-I_{jk} \, j}^{(m)}
	\, \varphi^{(m)}(\xi+2-2\nu_{jk},\eta_k) \right\}  \,
	 d\xi  \\
+ & \, \, \frac{1}{2} \sum_{k=1}^M \omega_k \,
	\int_{-1+2\nu_{jk}}^{1} \left\{ \sum_{m=1}^{M(M+1)/2}F_{i-I_{jk} \, j}^{(m)} \, \varphi^{(m)}(\xi-2\nu_{jk},\eta_k)  \right\} \,
	 d\xi.
\end{split}
\]
We  note that the terms inside the brackets are simply shifted 
solution values, allowing us to express the above update as follows:
\[
\begin{split}
F^{(1),\text{new}}_{ij}  
=  \, \, \frac{1}{2} \sum_{k=1}^M  \omega_k \, \left\{
 	\int_{-1}^{-1 + 2 \nu_{jk}}  \, P_{ijk}^{R}(\xi) \, d\xi  
+  \int_{-1+2\nu_{jk}}^{1}  P_{ijk}^{L}(\xi)  \,
	 d\xi \right\},
\end{split}
\]
where
\begin{alignat*}{2}
P_{ijk}^{R}(\xi) &:= \tilde{f}^h_{i-1-I_{jk} \, j} (t^n, \xi+2-2\nu_{jk}, \eta_k) \quad
&& \text{for} \quad \xi \in \left[ -1, -1 + 2 \nu_{jk} \right], \\
P_{ijk}^{L}(\xi) &:=  \tilde{f}^h_{i-I_{jk} \, j} (t^n, \xi-2\nu_{jk}, \eta_k)  \quad
&& \text{for} \quad \xi \in \left[-1 + 2 \nu_{jk}, 1 \right].
\end{alignat*}
Since $P_{ijk}^{L}(\xi)$ and $P_{ijk}^{R}(\xi)$ are polynomials of degree at most $M-1$,
we can exactly evaluate each of the above integrals via Gauss-Legendre quadrature
rules using $K$ points (where $K$ is defined in \eqref{eqn:Kceil}):
\[
\begin{split}
F^{(1),\text{new}}_{ij}  
=  \, \, \frac{1}{2} \sum_{k=1}^M  \omega_k \, \left\{
 	\sum_{\ell=1}^{K}  \varpi_{\ell} P^{R}_{ijk}\left(\xi^{R}_{\ell j k}\right)   
+  \sum_{\ell=1}^{K}  \varpi_{\ell}  P^{L}_{ijk}\left(\xi^{L}_{\ell j k}\right)   \right\},
\end{split}
\]
where the $\varpi_{\ell}$'s  are the standard quadrature weights 
for Gauss-Legendre quadrature with $K$ points.

To conclude our proof, we note that since all of the quadrature weights in the above 
expression for $F^{(1),\text{new}}_{ij}$
are strictly positive,  we obtain positivity in the mean,
\[
	F^{(1),\text{new}}_{ij}  \ge 0,
\]
if $\tilde{f}^h_{ij}(t^n,\xi,\eta)$ is non-negative at all of the $2MK$ points defined in 
\eqref{eqn:pos_pts_1}--\eqref{eqn:pos_pts_2}.
\end{proof}

\subsection{Positivity-preserving  limiter}
\label{sec:positivity}
One of the key assumptions in the above proof of positivity in
the mean is the fact that solution prior to a time-step must
be positive at all of points defined in \eqref{eqn:pos_pts_1}--\eqref{eqn:pos_pts_2}. 
We show in this subsection how to
the limit the solutions, including the initial condition, so that
we achieve positivity at all of these points.
The key piece of technology necessary for achieving
this positivity is a modification of the limiter of
Zhang and Shu \cite{article:ZhShu10}. This limiter is simple
to implement and is completely local to each element. 

The solution on some element ${\mathcal T}$ can be
written as
\begin{equation}
	f^h(\xi,\eta) := \sum_{\ell=1}^{M(M+1)/2}
		F^{(\ell)} \, \varphi^{(\ell)}(\xi, \eta),
\end{equation}
where $M$ is the desired order of accuracy in space.
We assume that the element average is non-negative:
$F^{(1)} \ge 0$.
We sample this solution on a set of {\it test points}:
\begin{equation}
  \left(\xi_i, \, \eta_i\right) \in [-1,1] \times [-1,1] \qquad 
  \text{for} \qquad  i=1,2,\ldots,P,
\end{equation}
and define:
\begin{equation}
\label{eqn:pos_min}
    m :=  \min_{i=1,\ldots,P} f^h(\xi_i, \, \eta_i).
\end{equation}
Note that $m \in (-\infty, F^{(1)}]$.

The {\it limited solution} is defined as follows:
\begin{equation}
\label{eqn:limit_pos1}
	\tilde{f}^h(\xi,\eta) := F^{(1)}  + \theta \cdot \left( f^h(\xi,\eta) - F^{(1)} \right),
\end{equation}
where
\begin{equation}
\label{eqn:limit_pos2}
 \theta = \min \left\{ 1, \, \frac{ F^{(1)}}{F^{(1)} - m} \right\}.
\end{equation}
Note that $0 \le \theta \le 1$ and that 
\begin{align}
	\theta = \begin{cases}
		1 \quad &\text{if} \quad  0 \le m \le F^{(1)}, \\
	       \in [0,1) \quad &\text{if} \quad m < 0.
	       \end{cases}
\end{align}
This means that if the solution is already non-negative at each
of the test points, then
this limiter does not alter the solution. On the other hand, if the
solution on the element is negative at any of the test points,
then the high-order corrections are {\it damped}
until the solution is again non-negative. We are guaranteed that as $\theta \rightarrow 0$,
the solution will eventually become non-negative on the entire element since
$F^{(1)} \ge 0$.
%
%

In practice we implement the limiting strategy as follows:
\begin{enumerate}
\item During each of the stages labelled 3--9 in Algorithm \ref{alg3},
we apply the positivity limiter with the test
 points given by  \eqref{eqn:pos_pts_1}--\eqref{eqn:pos_pts_2}.
In this case $P=2MK$, where $M$ is the desired
 order of accuracy in space and $K$ is defined by 
 \eqref{eqn:Kceil}.
As proved in Theorem \ref{thm:positivity}, this 
guarantees that in each stage the approximate solution
remains positive in the mean.

\smallskip

\item After all of the stages of Algorithm \ref{alg3} have been carried out,
we apply the positivity limiter one more time to the solution, 
this time with the test points taken as
the $P=M^2$ Gauss-Legendre quadrature points on $[-1,1] \times [-1,1]$.
This final limiting provides some additional positivity enforcement
and allows us to compute a variety of integrals of the form \eqref{eqn:int_functional}
with function values that are non-negative. This is particularly useful
in computing the $L_1$-norm \eqref{eqn:L1norm_num}, the
total energy  \eqref{eqn:energy_num}, and the entropy \eqref{eqn:entropy_num}.
\end{enumerate}

\section{Numerical examples}
\label{sec:numericaL_examples}
In this section we apply the proposed scheme to a
variety of numerical test cases.  
We begin in \S \ref{subsec:linadv}
by considering a linear advection equation with
a velocity field that produces solid body rotation. This example
is primarily used to show the benefits of switching from second
to fourth-order operator splitting strategies. In \S \ref{subsec:forced}
we verify the order of accuracy of the method on a forced Vlasov-Poisson
equation where we know the exact solution. In the subsequent
three subsection we consider three standard problems for the
Vlasov-Poisson system:  \S \ref{subsec:twostream} the two-stream
instability,  \S \ref{subsec:weaklandau} weak Landau damping,
and  \S \ref{subsec:stronglandau} strong Landau damping.
Unless otherwise stated, all simulations  below
 are done with $5^{\text{th}}$ order in space and with the positivity-preserving
 limiters as described in \S\ref{sec:positivity} turned on.

\subsection{Linear advection}
\label{subsec:linadv}
We first consider a linear advection under a divergence-free velocity field:
\begin{equation}
	q_{,t} + {\bf u} \cdot  \nabla q = 0.
\end{equation}
We take the computational domain to be $[0,1] \times [0,1]$ 
and the velocity field to be solid body rotation around $(0.5, 0.5)$:
\begin{equation}
 {\bf u} = \left( u(y), v(x) \right) = \left(  \pi \left( 2 y - 1 \right), \pi \left( 1 - 2 x  \right) \right)^T.
\end{equation}
The initial condition is taken to be a smooth, compactly supported bump that
is centered at $(x_0, y_0) = (0.4, 0.5)$:
\begin{equation}
  q(0,x,y) = \begin{cases}
  	\cos^6\left( \frac{5 \pi}{3} r \right) & \quad \text{if}
		\quad r \le 0.3, \\
	0 & \quad \text{otherwise},
  \end{cases}
\end{equation}
where 
\begin{equation}
r = \sqrt{(x-x_0)^2+(y-y_0)^2}.
\end{equation}
This problem, just as the Vlasov-Poisson system, is solved via 
operator splitting on the two operators:
\begin{align}
	\text{Problem ${\mathcal A}$:} & \quad q_{,t} + u(y) \, q_{,x} = 0, \\
	\text{Problem ${\mathcal B}$:} & \quad q_{,t} + v(x) \, q_{,y} = 0.
\end{align}
We run the initial condition out to time $t=1$, at which point it should
return to its initial state.
The errors are computed using the relative $L_2$ errors defined by equation
\eqref{eqn:L2error_2D} with $M=5$ and varying $\Delta x=\Delta y$.
See \ref{sec:L2error_2D} for more details.
Convergence studies with Strang and the fourth-order operator
splitting results are shown in Table \ref{table:Advection}.

\begin{table}
\begin{center}
\begin{Large}
\begin{tabular}{|c||c|c||c|c|}
\hline
{\normalsize \text{\bf Mesh}} & 
{\normalsize \text{\bf SL2 Error}} & {\normalsize\text{$\log_2${\bf (Ratio)}}} &
{\normalsize \text{\bf SL4 Error}} & {\normalsize\text{$\log_2${\bf (Ratio)}}} \\
\hline \hline 
{\normalsize $10^2$}  & {\normalsize $3.215\times 10^{-1}$} & {\normalsize --}    & {\normalsize $5.679\times 10^{-1}$}  & {\normalsize --} \\
\hline
{\normalsize $20^2$}  & {\normalsize $7.185\times 10^{-2}$} & {\normalsize 2.16}  & {\normalsize $3.113\times 10^{-2}$}  & {\normalsize 4.19} \\
\hline
{\normalsize $40^2$}  & {\normalsize $1.578\times 10^{-2}$} & {\normalsize 2.19}  & {\normalsize $1.691\times 10^{-3}$}  & {\normalsize 4.20} \\
\hline
{\normalsize $80^2$}  & {\normalsize $3.923\times 10^{-3}$} & {\normalsize 2.01}  & {\normalsize $1.010\times 10^{-4}$}  & {\normalsize 4.07}  \\
\hline
{\normalsize $160^2$} & {\normalsize $9.890\times 10^{-4}$} & {\normalsize 1.99}  & {\normalsize $6.220\times 10^{-6}$}  & {\normalsize 4.02} \\
\hline
{\normalsize $320^2$} & {\normalsize $2.454\times 10^{-4}$} & {\normalsize 2.01}  & {\normalsize $3.843\times 10^{-7}$}  & {\normalsize 4.02} \\
 \hline
{\normalsize $640^2$} & {\normalsize $6.136\times 10^{-5}$} & {\normalsize 2.00}  & {\normalsize $2.390\times 10^{-8}$}  & {\normalsize 4.01} \\
\hline
\end{tabular}
\end{Large}
\end{center}
\caption{Convergence study for the linear advection equation. Shown
are the relative errors computed via \eqref{eqn:L2error_2D}
at time $t=1$.  All calculations were done with $5^{\text{th}}$ order accuracy in space using the positivity-preserving limiters and
a CFL number of 5.00, where 
$\text{CFL} := \Delta t\max\left\{ \max_y|u(y)| / \Delta x, \,
\max_x|v(x)| / \Delta y\right\}$.
SL2 refers to the Strang split semi-Lagrangian scheme and
SL4 to the fourth-order split semi-Lagrangian scheme.\label{table:Advection}}

\end{table}

\subsection{A forced problem: verifying order of accuracy}
\label{subsec:forced}
Next we consider an example of a forced Vlasov-Poisson equation where we
have an exact solution.
The forced Vlasov-Poisson system is
\begin{gather}
f_{,t} + v f_{,x} + E f_{,v} = \psi(t,x,v), \\
E_{,x} =   \int_{-\infty}^{\infty} f(t,x,v) \, dv - \sqrt{\pi},
\end{gather}
on $(t,x,v) \in [0, \infty) \times [ -\pi, \pi ] \times (-\infty, \infty)$
with periodic boundary conditions in $x$.
We take the following source term:
\begin{equation}
\begin{split}
\psi(t,x,v) = \frac{1}{2} \sin(2x-2\pi t) e^{-\frac{1}{4}\left(4v-1\right)^2}
\Bigl\{ & \left( 2\sqrt{\pi}+1 \right) \left( 4 v - 2\sqrt{\pi} \right) \\ & - 
	\sqrt{\pi} \left( 4v-1 \right) \cos(2x - 2\pi t) \Bigr\}. 
\end{split}
\end{equation}
The exact solution in this case is
\begin{align}
f(t,x,v) &= \left\{ 2 - \cos\left(2x - 2\pi t \right) \right\} e^{-\frac{1}{4} \left(4 v- 1 \right)^2}, \\
E(t,x) &= -\frac{\sqrt{\pi}}{4} \, \sin\left( 2x - 2\pi t \right).
\end{align}

The numerical scheme for this forced problem is the same as the one
described in \S \ref{sec:numericaL_method} with two minor modifications.
First, the two operators in the operator split scheme are
\begin{align}
	\text{Problem ${\mathcal A}$:} & \quad f_{,t} + v \, f_{,x} = \psi(t,x,v), \\
	\text{Problem ${\mathcal B}$:} & \quad f_{,t} + E(t,x) \, f_{,v} = 0,
\end{align}
which means that Problem ${\mathcal A}$ is slightly modified
from the unforced Vlasov-Poisson system.
The modified  ${\mathcal A}$ still has the same characteristics
as the case with no source term; the only difference is that the
solution is no longer constant along the characteristics. 
In order to advance $f$ forward under the influence
of operator ${\mathcal A}$,  we use the method
of characteristics and obtain the following solution:
\begin{equation}
    f(t^{n+1},x,v) =  f(t^n,x-v\Delta t,v) + 
    \int_{t^n}^{t^{n+1}} \psi(s, x+v(s-t^{n+1}),v) \, ds.
\end{equation}
The time integral in the above expression can be easily
exactly evaluated; we omit the details here.

The second modification comes from the fact
that with a non-zero source term,
it is no longer true that $E_{,t} = - \rho u$. This means that the
time interpolation described in \S 
\ref{sec:high_order_semiL_alg} must be slightly modified.
Instead of using $E_{,t} = - \rho u$, \eqref{eqn:Ett}, and \eqref{eqn:Ettt},
we make use of the following modified formulas:
\begin{align}
E_{,t}(t,x) &= -\rho u + C_1, \\
E_{,tt}(t,x) &=  {\mathbb E}_{,x} - \rho  E + C_2, \\
E_{,ttt}(t,x) &=  \left( 2\rho u E \right)_{,x} - {\mathbb F}_{,xx} + E \left( \rho u \right)_{,x}
	+ \rho^2 u  + C_3,
\end{align}
where
\begin{align}
C_1 &:= \frac{\sqrt{\pi}}{4}+ \frac{\sqrt{\pi}}{8} (4 \pi-1) \cos(2 x-2 \pi t),\\
C_2 &:= \frac{3 \sqrt{\pi}+4 \pi-16 \sqrt{\pi^5}}{16}  \sin(-2 x+2 \pi t)
+\frac{\pi}{16}  \sin(4 x-4 \pi t),\\
C_3 &:= -\frac{\pi}{4} + \frac{7 \sqrt{\pi}+16 \pi-64 \sqrt{\pi^7}}{32} 
 \cos(2 x-2 \pi t)-\frac{3 \pi}{16} \cos(4 x-4 \pi t).
\end{align}
In the above expression we used the shorthand notation:
\[
	x := x^1, \quad E := E^1, \quad
	u := u^1, \quad {\mathbb E} := {\mathbb E}^{11},
	\quad \text{and} \quad {\mathbb F} := {\mathbb F}^{111}.
\]

We run the initial condition out to time $t=1$, at which point it should
return to its initial state.
Convergence studies on various grids on the domain
$(x,v) \in \left[-\pi, \pi \right] \times \left[-\pi, \pi \right]$ with Strang and the fourth-order operator
splitting results are shown in Table \ref{table:VPsource}.  
We compute
the errors in an identical manner to the linear test problem presented
in the previous section using the relative $L_2$ errors defined by equation
\eqref{eqn:L2error_2D} with $M=5$.
See \ref{sec:L2error_2D} for more details.

\begin{table}
\begin{center}
\begin{Large}
\begin{tabular}{|c||c|c||c|c|}
\hline
{\normalsize \text{\bf Mesh}} & 
{\normalsize \text{\bf SL2 Error}} & {\normalsize\text{$\log_2${\bf (Ratio)}}} &
{\normalsize \text{\bf SL4 Error}} & {\normalsize\text{$\log_2${\bf (Ratio)}}} \\
\hline \hline 
{\normalsize $10^2$}   & {\normalsize $5.210\times 10^{-1}$} & {\normalsize --}    & {\normalsize $9.493\times 10^{-1}$}  & {\normalsize --} \\
\hline
{\normalsize $20^2$}   & {\normalsize $1.433\times 10^{-1}$} & {\normalsize 1.86}  & {\normalsize $2.715\times 10^{-1}$}  & {\normalsize 1.81} \\
\hline
{\normalsize $40^2$}   & {\normalsize $1.640\times 10^{-2}$} & {\normalsize 3.13}  & {\normalsize $1.652\times 10^{-2}$}  & {\normalsize 4.04} \\
\hline
{\normalsize $80^2$}   & {\normalsize $3.438\times 10^{-3}$} & {\normalsize 2.26}  & {\normalsize $7.079\times 10^{-4}$}  & {\normalsize 4.55}  \\
\hline
{\normalsize $160^2$}  & {\normalsize $8.333\times 10^{-4}$} & {\normalsize 2.04}  & {\normalsize $3.434\times 10^{-5}$}  & {\normalsize 4.37} \\
\hline
{\normalsize $320^2$}  & {\normalsize $2.068\times 10^{-4}$} & {\normalsize 2.01}  & {\normalsize $1.962\times 10^{-6}$}  & {\normalsize 4.13} \\
 \hline
{\normalsize $640^2$}  & {\normalsize $5.161\times 10^{-5}$} & {\normalsize 2.00}  & {\normalsize $1.203\times 10^{-7}$}  & {\normalsize 4.03} \\
\hline
{\normalsize $1280^2$} & {\normalsize $1.290\times 10^{-5}$} & {\normalsize 2.00}  & {\normalsize $7.509\times 10^{-9}$}  & {\normalsize 4.00} \\
\hline

\end{tabular}
\end{Large}
\end{center}
\caption{Convergence study for the forced Vlasov-Poisson equation.  
All calculations presented here are $5^{\text{th}}$ 
order in space and were run with a CFL number of 5. 
Shown are the relative errors computed via \eqref{eqn:L2error_2D}
at time $t=1$.  SL2 refers to the Strang split semi-Lagrangian scheme and
SL4 to the fourth-order split semi-Lagrangian scheme.
Since the positivity-preserving limiters as described in 
\S \ref{sec:positivity} don't guarantee positivity in the mean in the 
presence of a source term, we have turned them off for this convergence
study only.
\label{table:VPsource}}

\end{table}

\subsection{Two-stream instability}
\label{subsec:twostream}
The two-stream instability problem has become a standard
benchmark to test numerical Vlasov solvers, and has been
used as such by several authors (e.g., 
\cite{article:Heath10,article:BanksHitt10,
article:QiuCh10,article:FiSo03,christlieb2006gfp,article:ChKn76}).
We use the following initial distribution function
\begin{equation}
f(t=0,x,v) = \frac{v^2}{\sqrt{8\pi}} \left( 2 - \cos
	\left(\frac{x}{2}\right) \right) e^{-\frac{v^2}{2}},
\end{equation}
and solve on the domain $(x,v) \in \left[-2\pi, 2\pi \right] \times
\left[-2\pi, 2\pi \right]$.  Results for time $t=45$ are presented in
Figure \ref{fig:two-stream} for various mesh sizes.  
In Figure \ref{fig:two-stream-slice} we present vertical cross-sections
of the solution taken at $x=0$ for the same mesh sizes.  

Figures \ref{fig:two-stream} and \ref{fig:two-stream-slice} clearly
demonstrate the ability of the discontinuous Galerkin methodology
to approximate very rough data, something that is more difficult with
methods that act over larger stencils. The results shown in these
figures indicate far more structure than what is shown in
other recent work, including results from the WENO method \cite{article:QiuCh10, article:BanksHitt10}
and an explicit DG method that uses a piecewise constant
representation of the distribution function, $f$, and a piecewise
quadratic representation of the electric potential $\phi$ \cite{article:Heath10}.

In Figure 
\ref{fig:two-stream-slice-neg} we demonstrate the effects of adding the
posivitity-preserving limiter. We see that even without limiting, the base scheme 
already does a reasonable job of not producing large negative
values in the distribution function. With the positivity-preserving limiters
we are able to remove these small positivity violations.
In Figure \ref{fig:two-stream-consv} we plot four quantities that are exactly
conserved by the continuous Vlasov-Poisson equation, but only
approximately conserved in our numerical discretization:
$L_1$-norm of $f$ \eqref{eqn:L1norm}, $L_2$-norm of $f$ \eqref{eqn:L2norm}, 
total energy \eqref{eqn:energy}, and
total entropy \eqref{eqn:entropy}.
In particular, we use the numerical approximations to 
\eqref{eqn:L1norm}--\eqref{eqn:entropy} as given by equations 
\eqref{eqn:L1norm_num}--\eqref{eqn:entropy_num} in 
\ref{sec:num_int}.
We note that it is difficult to obtain accurate values for the total entropy \eqref{eqn:entropy},
because there are many values where $f$ becomes very small.

\begin{figure}
\begin{center}
   \includegraphics[width=60mm]{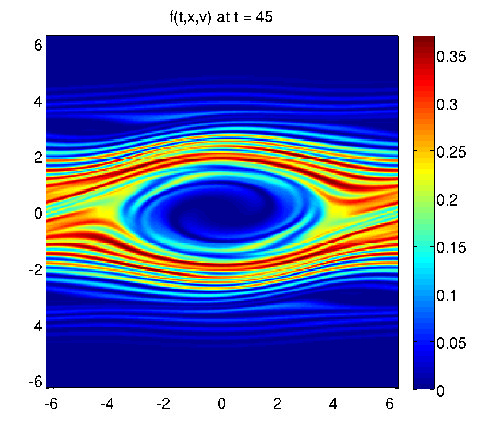}
   \includegraphics[width=60mm]{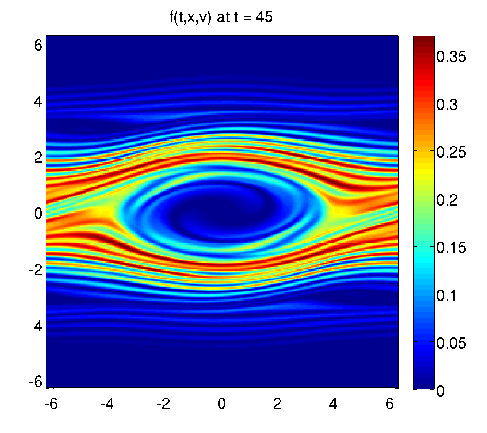}
   
   \vspace{2mm}
   
   \includegraphics[width=60mm]{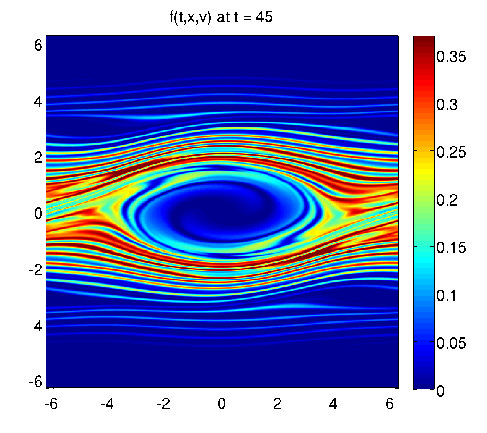}
   \includegraphics[width=60mm]{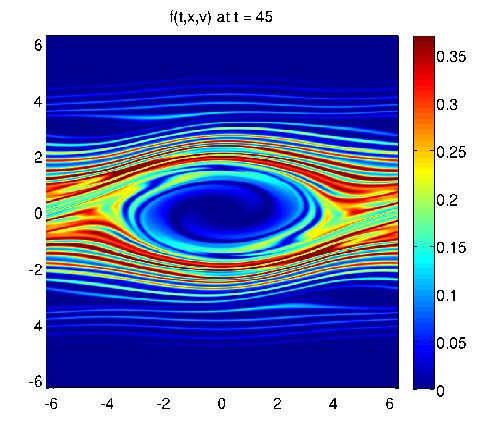}
   
   \vspace{2mm}

   \includegraphics[width=60mm]{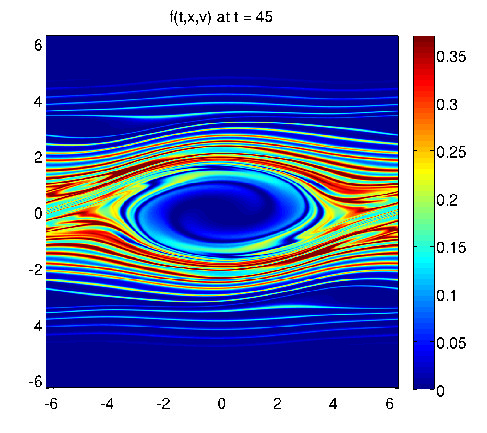}
   \includegraphics[width=60mm]{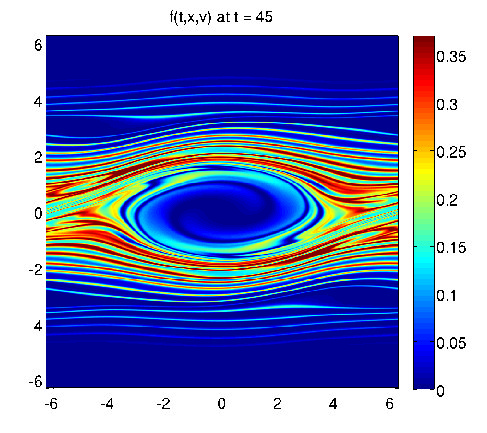}

  \caption{
  The two-stream instability problem.
 The panels in the left-hand column are results using the $2^{\text{nd}}$
    order Strang splitting method.
    The panels in the right-hand column are results using the $4^{\text{th}}$
    order splitting method.
     All simulations are
    $5^{\text{th}}$ order in space.
    The mesh sizes for the first, second and third rows are 
    \mbox{$(m_x, m_v) = (65, 65)$},
    \mbox{$(m_x, m_v) = (129, 129)$}, and
    \mbox{$(m_x, m_v) = (255, 255)$}, respectively.  
    All solutions use the
    positivity-preserving algorithm. The above figures were produced by plotting
    the numerical solution at each of the $5 \times 5$ Gaussian quadrature points in
    each mesh element.
	\label{fig:two-stream}}
\end{center}
\end{figure}

\begin{figure}
\begin{center}
   \includegraphics[width=58mm]{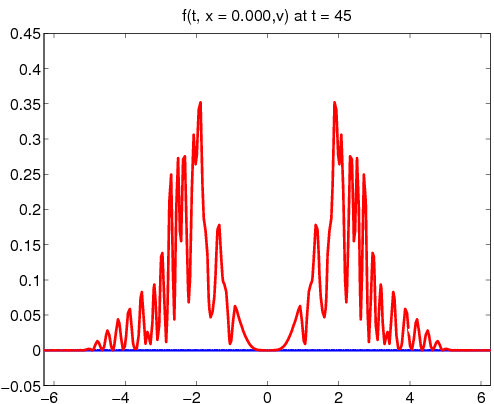} \quad
   \includegraphics[width=58mm]{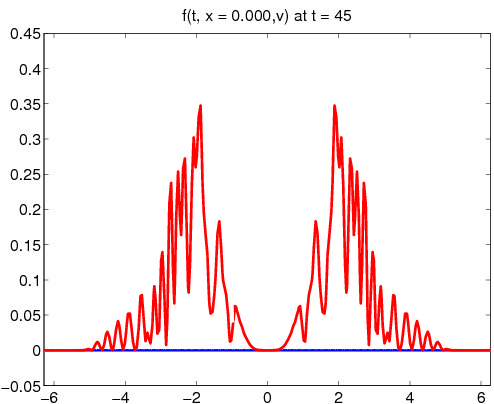}
   
   \vspace{4mm}
   
   \includegraphics[width=58mm]{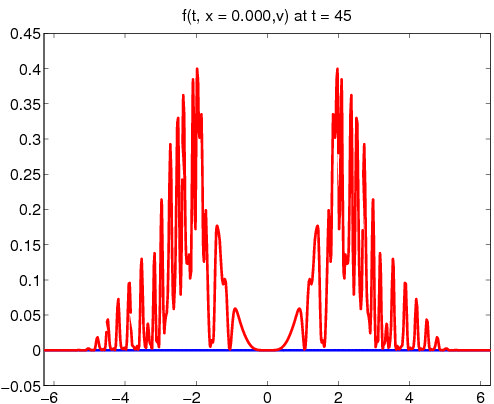} \quad
   \includegraphics[width=58mm]{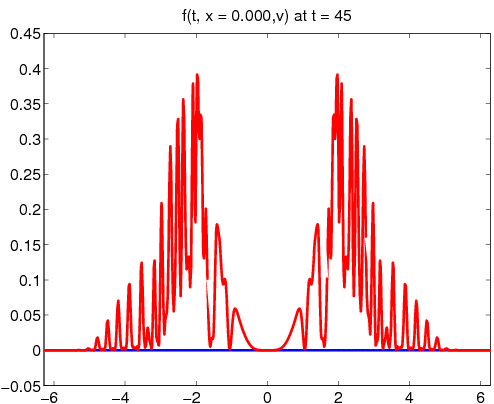}
   
   \vspace{4mm}

   \includegraphics[width=58mm]{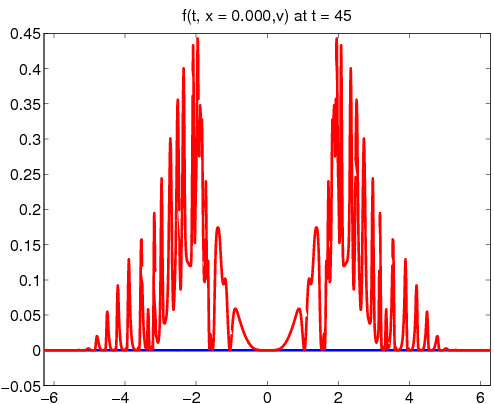} \quad
   \includegraphics[width=58mm]{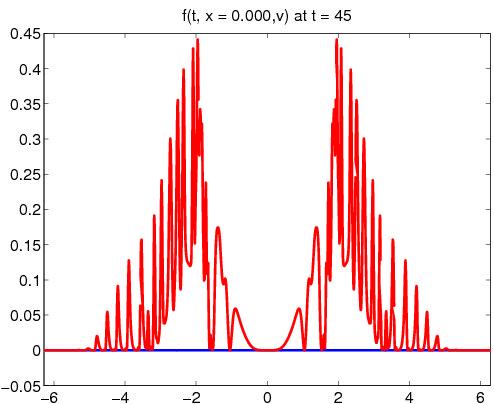}

  \caption{The two-stream instability problem.
    The panels in the left-hand column are results using the $2^{\text{nd}}$
    order Strang splitting method.
    The panels in the right-hand column are results using the $4^{\text{th}}$
    order splitting method.
     All simulations are
    $5^{\text{th}}$ order accurate in space.
    The mesh sizes for the first, second and third rows are 
    \mbox{$(m_x, m_v) = (65, 65)$},
    \mbox{$(m_x, m_v) = (129, 129)$}, and
    \mbox{$(m_x, m_v) = (255, 255)$}, respectively.  
    All solutions use the
    positivity-preserving algorithm. Each figure above represents
    the numerical solution at $x=0$ and use 5 Gaussian quadrature points 
    for each cell in the $v$-coordinate.
    The above solutions were computed with an odd number of elements
    in each coordinate direction in order to easily obtain a slice of the solution at $x=0$.  
	\label{fig:two-stream-slice}}
\end{center}
\end{figure}

\begin{figure}
\begin{center}
   \includegraphics[width=58mm]{outputSL4_129_slice.jpg} \quad
   \includegraphics[width=58mm]{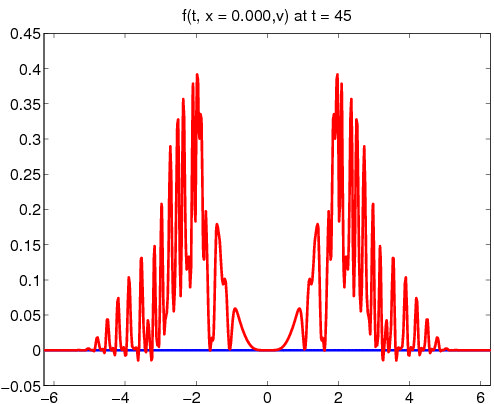}

  \caption{The two-stream instability problem.
    These panels demonstrate the effect of the positivity preserving-limiter.
    The result on the left is 
    the $4^{\text{th}}$ order splitting algorithm with limiters turned on,
	and the result on the right hand side is the same algorithm with the limiters turned off.  
	Both pictures represent a slice of the solution
	at $x = 0$ and final time $t=45$.  
	Both results represent a mesh of size  $(m_x, m_v) = (129, 129)$ and are
	$5^{\text{th}}$ order accurate in space; an odd number of grid elements 
    are used in order to easily obtain function values.
	We further note that $\min_i f(x_i,v_i) = -2.020 \times 10^{-2}$ 
	for the solution without the limiter and $\min_i f(x_i,v_i) = 7.000 \times 10^{-12}$ 
    for the limited solution, where the minimum
	is taken over all $5\times5$ Gaussian quadrature points over every mesh element.  
    \label{fig:two-stream-slice-neg}}
\end{center}
\end{figure}

\begin{figure}
\begin{center}
   \includegraphics[width=58mm]{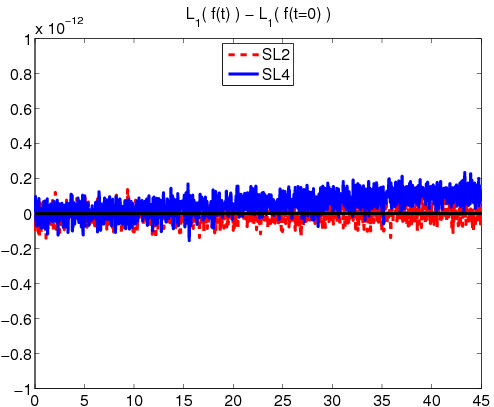} \quad
   \includegraphics[width=58mm]{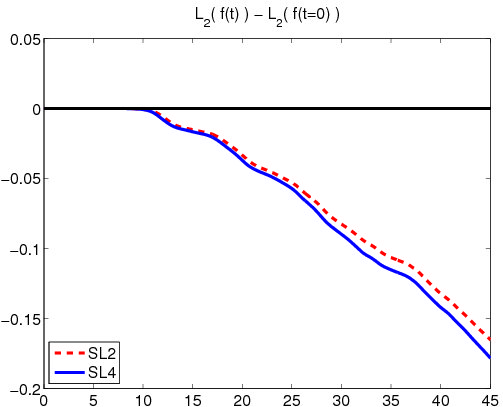} 

   \vspace{4mm}

   \includegraphics[width=58mm]{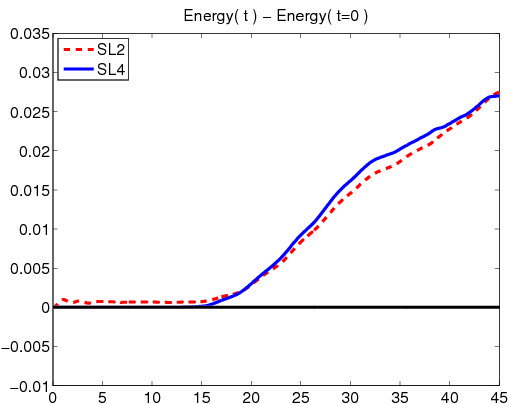} \quad
   \includegraphics[width=58mm]{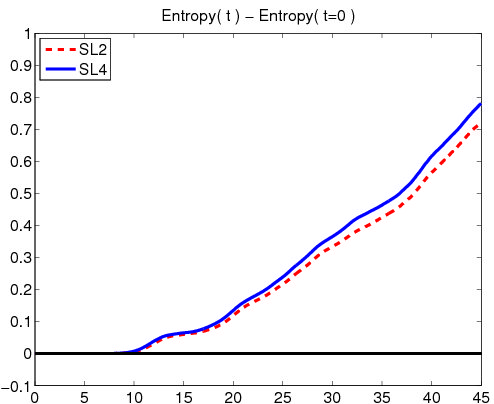} 

  \caption{
    The two-stream instability problem.
	Shown in these panels are the deviations of the $L_1$ 
    norm of $f$ (top-left), $L_2$ norm of $f$ (top-right), 
	total energy (bottom-left), and total entropy (bottom-right) from their
	initial values.  All simulations 
	use a constant CFL number of 2.0, and are obtained from a mesh of size
    $(m_x, m_v) = (129,129)$. The domain for this problem is
	$(x,v) \in [-2\pi, 2\pi] \times [-2\pi,2\pi]$.
	Each simulation is $5^{\text{th}}$ order accurate in space
    and is positivity preserving.
	\label{fig:two-stream-consv}}
\end{center}
\end{figure}

\subsection{Weak Landau damping}
\label{subsec:weaklandau}
Landau damping has been extensively studied both numerically
\cite{article:ChKn76,article:Heath10, article:Zhou01} and
 analytically (e.g., the work of Mouhot and Villani
\cite{article:MoVi10}). Just as the two-stream instability problem,
Landau damping has become a favorite standard test case. It is particularly
useful since the
linear decay rates of the $L_2$-norm of the electric field are
well-known.

We use the following initial distribution function
\begin{equation}
\label{eqn:landau_damping}
	f(t=0,x,v) = \frac{1}{\sqrt{2\pi}} \Bigl( 1 + \alpha
		\cos(k x) \Bigr) e^{-\frac{v^2}{2}},
\end{equation}
with $\alpha = 0.01$ and $k = 0.5$ on the domain
$(x,v) \in \left[-2\pi, 2\pi \right] \times \left[-2\pi, 2\pi \right]$.  Because
$\alpha$ is a small parameter, we expect to see results that closely agree with
the linear theory, where the electric field decays exponentially.
In Figure \ref{fig:weak-Efield} we present this decay provided by 
\[
	\log\left( \| E( t, \cdot) \|_{L_2} \right) := \log\sqrt{\int_{-2\pi}^{2\pi} \left|E(t,x)\right|^2 \ dx}
\]
versus time for two different mesh sizes. 
Our computed decay rate matches the linear decay rate, $\gamma=-0.1533$.
In Figure \ref{fig:weak-consv} 
we again plot the deviations of several quantities that are conserved
by the continuous Vlasov-Poisson system from their
initial values: $\|f\|_{L_1}, \|f\|_{L_2}$, total energy, and entropy.
We again use the numerical approximation to these norms given by
\eqref{eqn:L1norm_num}--\eqref{eqn:entropy_num} in 
\ref{sec:num_int}.
Our results are comparable to what is reported for example by Qiu
and Christlieb \cite{article:QiuCh10}.

\begin{figure}
\begin{center}
   \includegraphics[width=58mm]{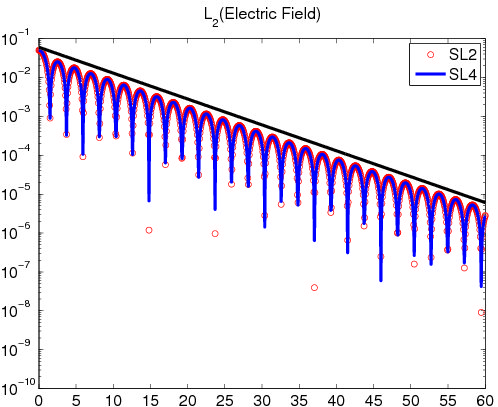} \quad
   \includegraphics[width=58mm]{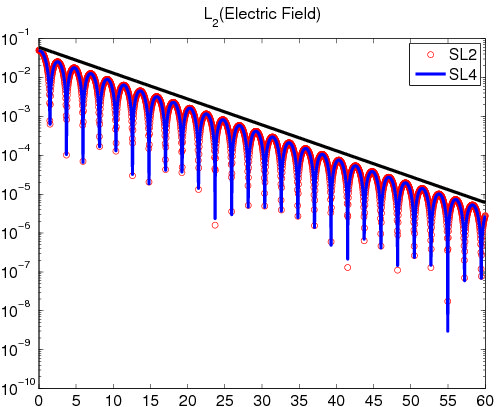}

  \caption{The weak Landau damping problem.
	Shown in the panels are semi-log plots of the $L_2$-norm of the electric field. 
    Simulations use a constant CFL number of 2.0
	and are $5^{\text{th}}$ order accurate in space.  All simulations
	use the positivity-preserving limiter.
	The figure on the left represents a mesh of size 
    \mbox{$(m_x,m_v) = (64,128)$} 
    and the result on the right was represents a mesh
	of size 
    \mbox{$(m_x, m_v) = (128,256)$}, both on the
    domain $(x,v) \in [-2\pi,2\pi] \times [-2\pi,2\pi]$.  
    Both simulations match the theoretical decay rate of $\gamma = -0.1533$, and to demonstrate
	this we plot the line defined by $y = 0.06 e^{\gamma t}$.
	One should note that the discrepancy in the two plots is due to the fact
	that twice as many time points in the second plot as the first one.
	\label{fig:weak-Efield}}
\end{center}
\end{figure}

\begin{figure}
\begin{center}
   \includegraphics[width=50mm]{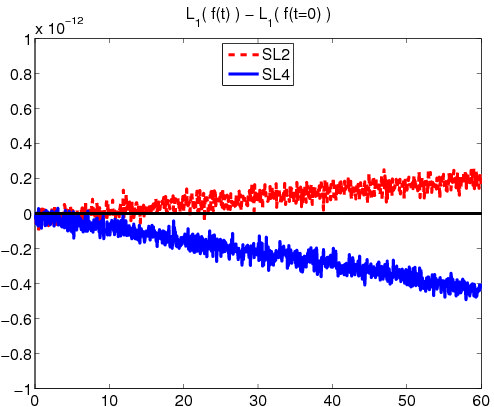} \qquad
   \includegraphics[width=50mm]{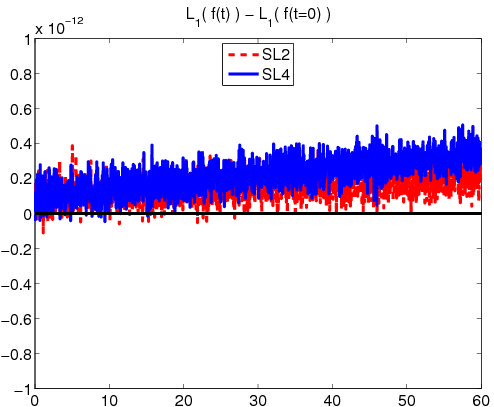} 

   \vspace{2mm}

   \includegraphics[width=50mm]{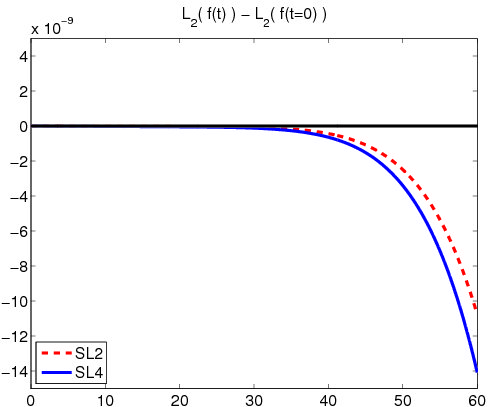} \qquad
   \includegraphics[width=50mm]{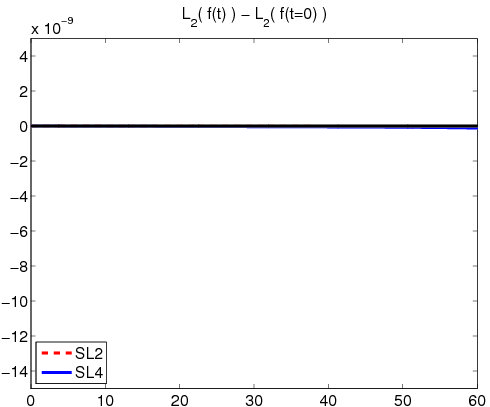} 

   \vspace{2mm}
   
   \includegraphics[width=50mm]{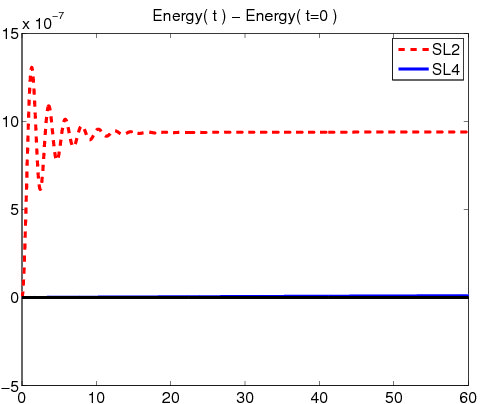} \qquad
   \includegraphics[width=50mm]{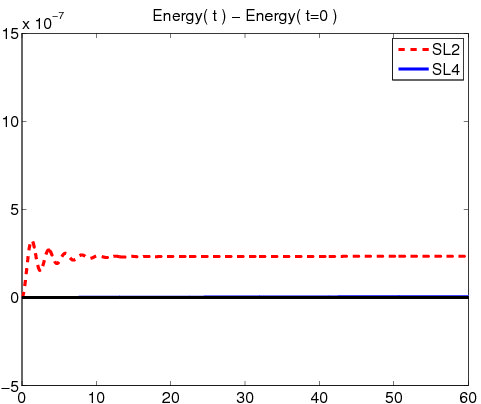}

   \vspace{2mm}

   \includegraphics[width=50mm]{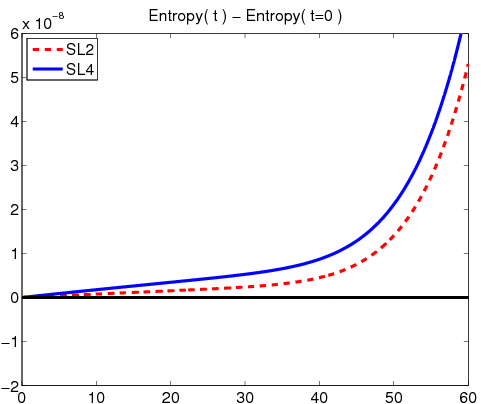} \qquad
   \includegraphics[width=50mm]{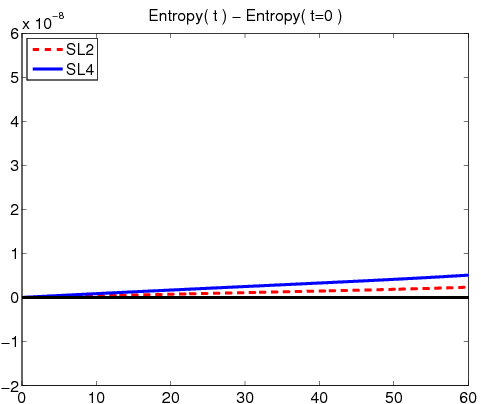}

  \caption{The weak Landau damping problem. 
     Shown in the panels are the deviations of various quantities that
     are conserved by the exact equations from their
     initial values.  All simulations use a constant CFL number 
     of 2.0, are $5^{\text{th}}$ order accurace in space, and are positivity
     preserving.
     The mesh size for the left hand column is 
     \mbox{$(m_x, m_v) = (64, 128)$} and the mesh size for the right hand column
     is \mbox{$(m_x, m_v) = (128, 256)$}.  We note that the largest
     deviation for total energy for the $4^{\text{th}}$ order algorithm is on the 
     order of $10^{-10}$.
	\label{fig:weak-consv}}
\end{center}
\end{figure}

\subsection{Strong Landau damping}
\label{subsec:stronglandau}
The initial condition is again given by \eqref{eqn:landau_damping},
this time with $\alpha = 0.5$ and $k = 0.5$ on the domain
 $(x,v) \in \left[-2\pi, 2\pi \right] \times
\left[-2\pi, 2\pi \right]$.  The time evolution of the distribution function is shown in the
panels of Figure \ref{fig:strong-2dpics}.
These images are comparable to what is shown in Qiu and Christlieb
\cite{article:QiuCh10}, but we are again able to capture more fine scale structure with
the discontinuous Galerkin approach.

A semi-log plot of the $L_2$-norm of the electric field is provided in Figure \ref{fig:strong-Efield}, and
decay rates are computed by sampling the solution at data points.  We find that
the initial linear decay rate is approximately
$\gamma_1 \approx -0.292$ which is close to the value of $-0.243$ computed by
Zaki et al. \cite{article:Zaki88}, closer still to the value of $-0.281$ computed by 
Cheng and Knorr \cite{article:ChKn76}, but much larger than the value of $-0.126$
computed by Heath et. al  \cite{article:Heath10}.  
In this same figure we also estimate the growth rate due to
particle trapping and find it to be approximately $\gamma_2=0.0815$;
this number also differs from the value reported by Heath et al. \cite{article:Heath10}:
$\gamma_2 = 0.0324$.
The initial linear decay was computed by taking the
maximum of the first two peaks located at $t\approx 2.45$ and $t \approx 4.54$.  
For the particle trapping growth regime, we sampled the 
maximum of the solution at the two peaks located at 
$t\approx 2.33$ and $t \approx 2.84$.
We postulate that the
difference between our computed growth rates
 and those of Heath et al. \cite{article:Heath10}
stems from the fact that we are using piecewise quartic polynomials to
represent the distribution function, while they are using only piecewise
constants. This issue should be further investigated.

 In Figure \ref{fig:strong-consv64} 
we again plot the deviations of several quantities that are conserved
by the continuous Vlasov-Poisson system from their
initial values: $\|f\|_{L_1}, \|f\|_{L_2}$, total energy, and entropy.
In particular, we use the numerical approximations to 
\eqref{eqn:L1norm}--\eqref{eqn:entropy} as given by equations 
\eqref{eqn:L1norm_num}--\eqref{eqn:entropy_num} in 
\ref{sec:num_int}.
Our results are comparable to what is reported for example by Qiu
and Christlieb \cite{article:QiuCh10}.

Finally, we note that in both our weak and strong Landau damping
simulations we made $V_{\text{max}}=2\pi$ instead of the more commonly used
value of $V_{\text{max}}=5$. The reason for  this is that we noticed
that between roughly $v=5$ and $v=6$ the distribution function $f(t,x,v)$
still had a non-negligable amplitude on the order of about $10^{-6}$; the precise
behavior of strong and weak Landau damping in this region is shown in 
Figure \ref{fig:strong-slice}. Therefore,
truncating at $V_{\text{max}}=5$ caused additional errors when tracking
various conserved quantities; we found improvements in these errors
when taking $V_{\text{max}}=2\pi$.




\begin{figure}
\begin{center}
   \includegraphics[width=41mm]{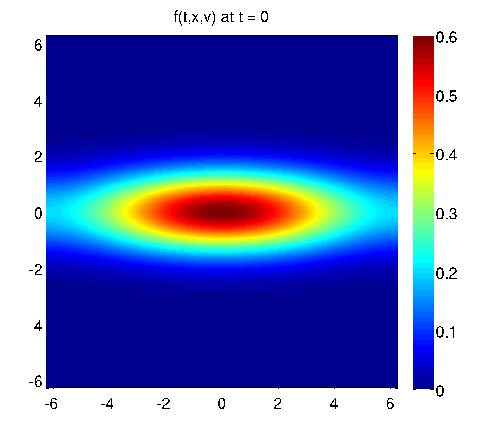} \hspace{-4mm}
   \includegraphics[width=41mm]{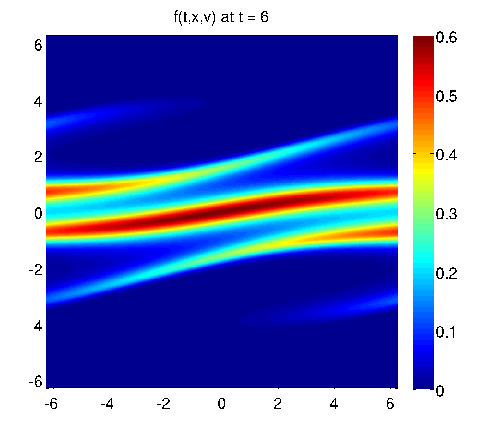} \hspace{-4mm}
   \includegraphics[width=41mm]{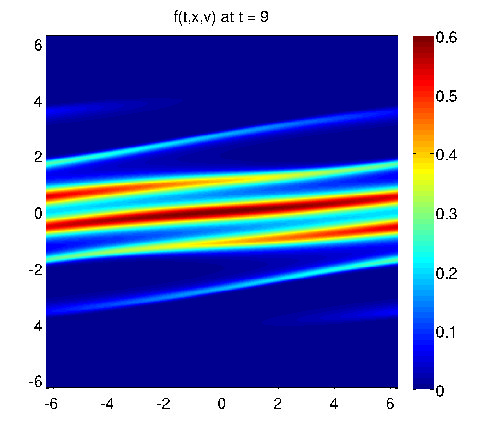} 

   \vspace{2mm}

   \includegraphics[width=41mm]{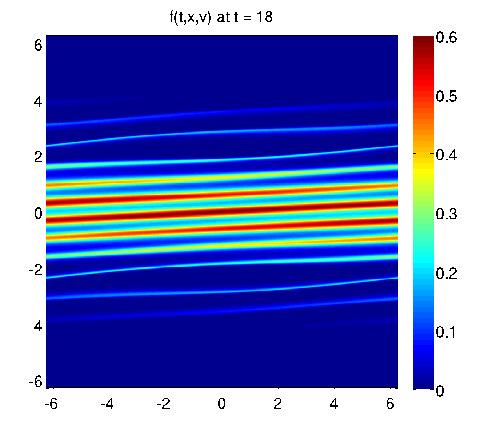} \hspace{-4mm}
   \includegraphics[width=41mm]{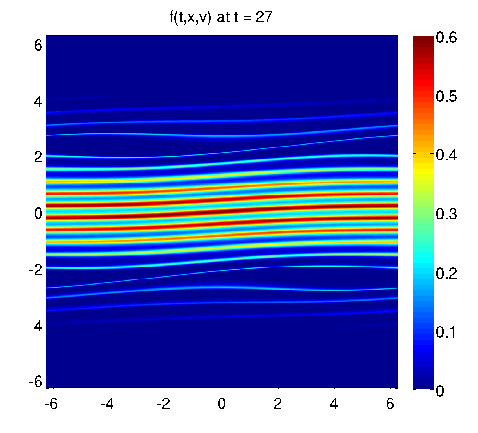} \hspace{-4mm}
   \includegraphics[width=41mm]{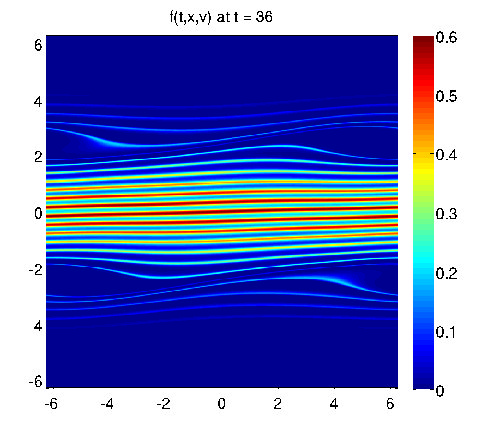} 
   
   \vspace{2mm}

   \includegraphics[width=41mm]{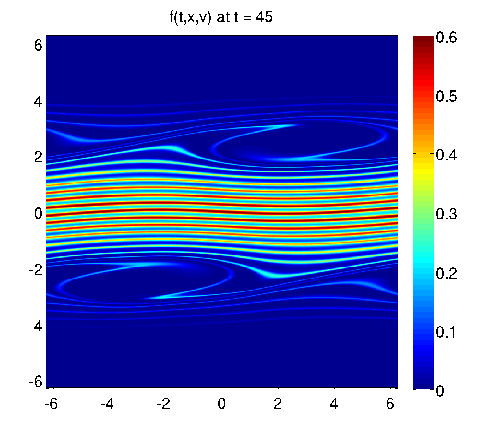} \hspace{-4mm}
   \includegraphics[width=41mm]{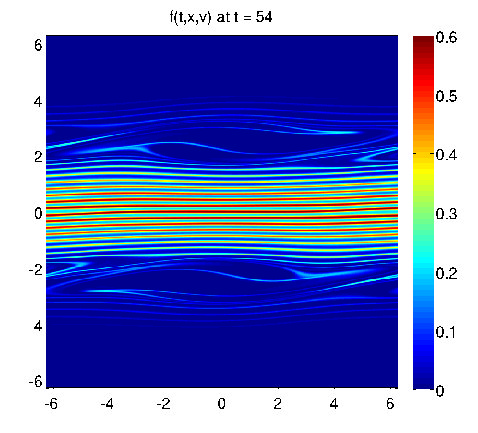} \hspace{-4mm}
   \includegraphics[width=41mm]{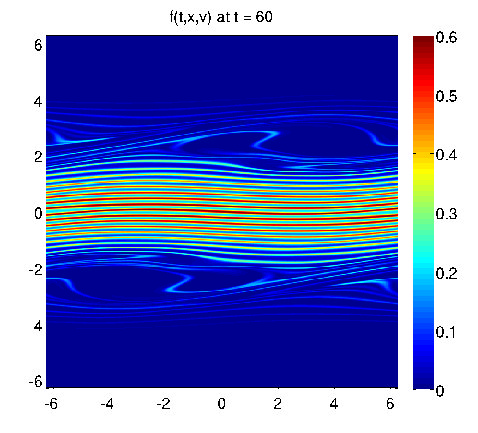} 

  \caption{The strong Landau damping problem. Shown in the panels
  are  the distribution function at various points in time.
	This simulation was run with a constant CFL number of 2.0 
    on a mesh of size
	\mbox{$(m_x, m_v) = (128, 256)$} using
	 $5^{\text{th}}$ order accuracy in space and
	  the positivity-preserving limiters.  It is clear from these
	  plots that the high-order
	discontinuous Galerkin method is able to capture much of the fine-scale
	structure for the
	solution. \label{fig:strong-2dpics}}
\end{center}
\end{figure}

\begin{figure}
\begin{center}
   \includegraphics[width=58mm]{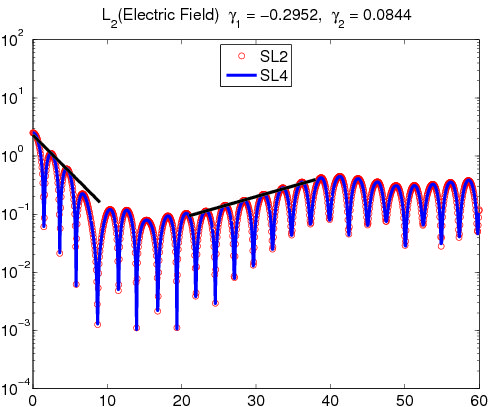} \quad
   \includegraphics[width=58mm]{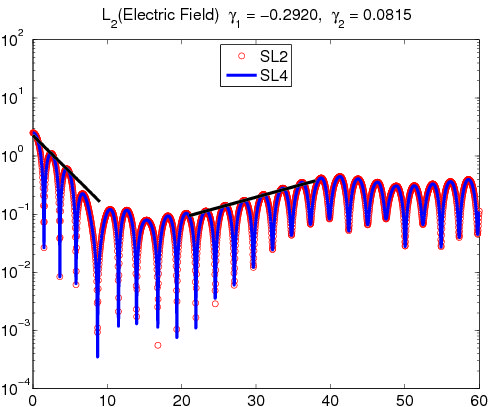}

  \caption{The strong Landau damping problem.
	Shown in these panels are semi-log plots of the $L_2$ norm of the 
    electric field with two different
	resolutions; the mesh size for the the figure on the left is
   {$(m_x, m_v) = (64,128)$} and the 
	figure on the right is
	{$(m_x, m_v) = (128, 256)$.} 
	Both simulations use the positivity preserving limiter, are 
    $5^{\text{th}}$ order accuracy in space and use
    a constant CFL number of $2.0$.
	In each panel, $\gamma_1$ refers
    to the slope of the initial decay, and $\gamma_2$ refers to the growth rate
    between times $t=20$ and $t=40$.
	\label{fig:strong-Efield}}
\end{center}
\end{figure}

\begin{figure}
\begin{center}
   \includegraphics[width=54mm]{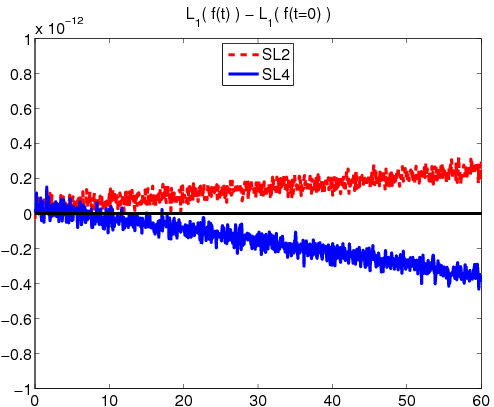} \qquad
   \includegraphics[width=54mm]{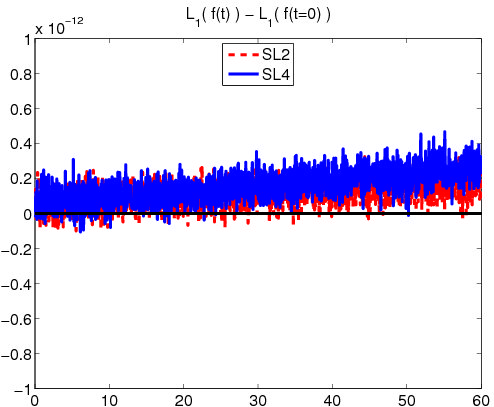}

   \includegraphics[width=54mm]{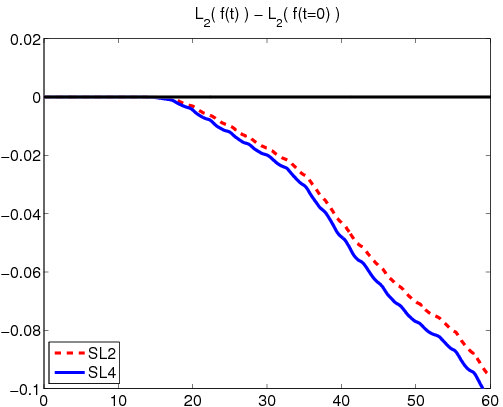} \qquad
   \includegraphics[width=54mm]{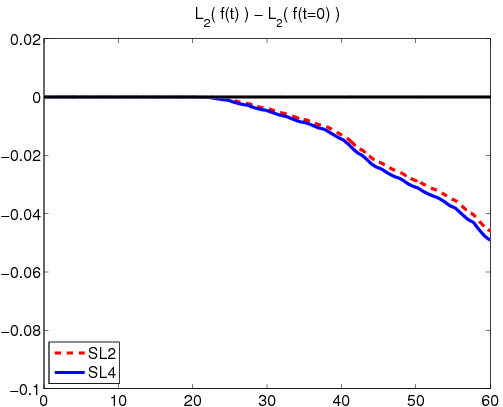}

   \includegraphics[width=54mm]{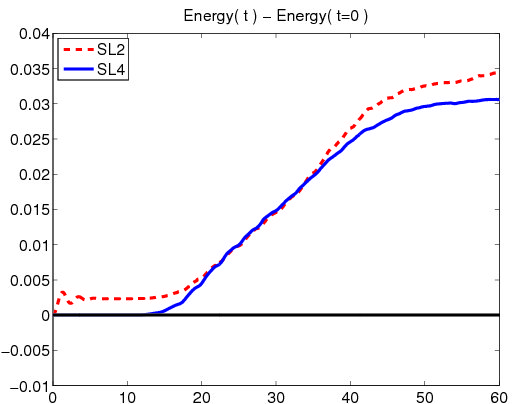} \qquad
   \includegraphics[width=54mm]{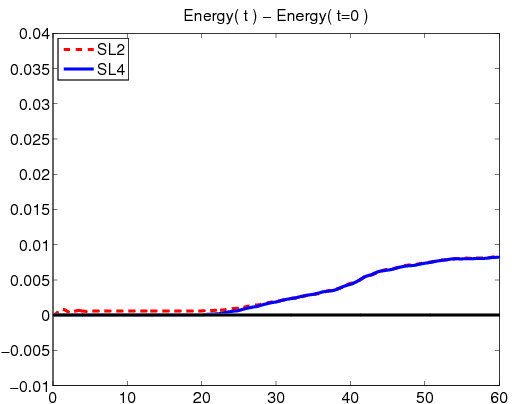}

   \includegraphics[width=54mm]{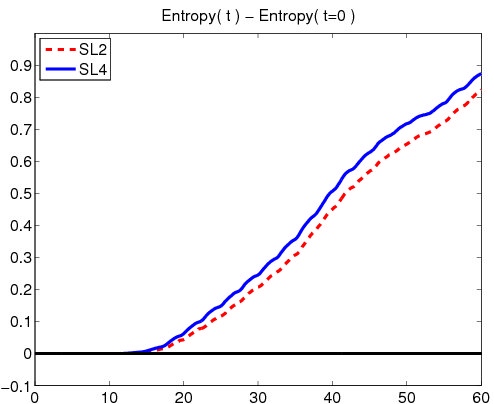} \qquad
   \includegraphics[width=54mm]{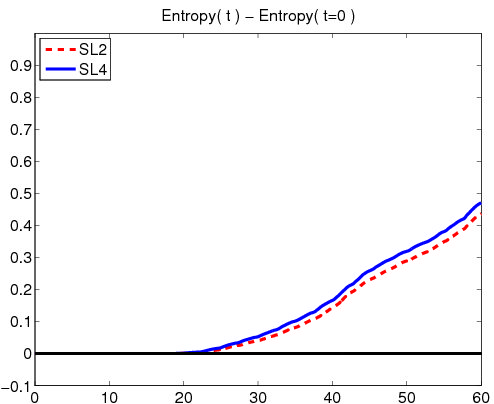} 

  \caption{The strong Landau damping problem.
	Simulation results for the $L_1$ norm (first row), $L_2$ norm (second row), 
	energy (third row), and entropy (bottom row)
	for strong Landau damping.  All simulations use a constant 
	CFL number of $2.0$ and are $5^{\text{th}}$ order accurate in space.  
	The mesh size for the left column
	is \mbox{$(m_x,m_v) = (64,128)$} and the mesh size for 
	the right column is \mbox{$(m_x, m_v) = ( 128, 256)$.}
	\label{fig:strong-consv64}}
\end{center}
\end{figure}

\begin{figure}
\begin{center}
   \includegraphics[width=58mm]{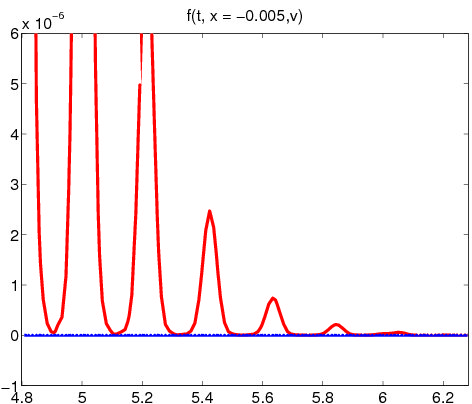} \quad
   \includegraphics[width=58mm]{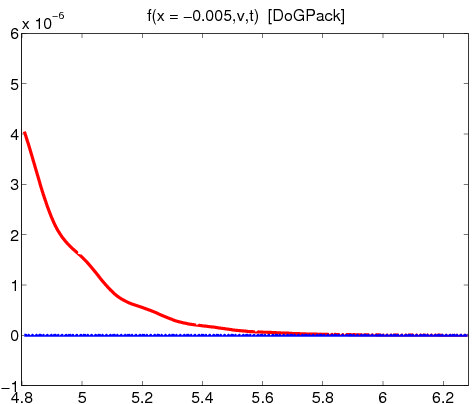}

  \caption{
    A comparison of vertical slices for strong landau damping (left) and 
    weak Landau damping (right) at time $t=60$.  The simulations for each 
    panel use a grid resolution
    of \mbox{$(m_x, m_v) = (128,256)$}, each are 
    $5^{\text{th}}$ order accuracy in space, and
    each use the positivity-preserving limiter.
    The CFL number for each simulation is $2.0$.  
    We note that the solution is non-zero for $|v| > 5$ in both cases, although
    there is much more activity in the case of strong Landau damping.  
    The plots suggests that the commonly used maximum velocity of $|v|=5$
    should be increased in order to get better accuracy in conservation of
    the quantities \eqref{eqn:L1norm_num}--\eqref{eqn:entropy_num}.
    \label{fig:strong-slice}}
\end{center}
\end{figure}

\section{Conclusions and future work}
\label{sec:conclusions}
We have described in this work a semi-Lagrangian discontinuous
Galerkin method for solving the 1+1 Vlasov-Poisson system.
This method was shown to have all of the following properties:
\begin{enumerate}
\item {\it Unconditionally stable};
 \item {\it High-order accurate in space} ($5^{\text{th}}$ order);
 \item {\it High-order accurate in time} ($4^{\text{th}}$ order);
\item {\it Mass conservative}; and
\item {\it Positivity-preserving}.
\end{enumerate}
The proposed method is based on a series quasi-1D semi-Lagrangian
advection steps coupled with a fourth-order operator splitting scheme.
The spatial discretization is handled via high-order discontinuous
Galerkin representations. The Poisson equation is solved to high-order
via a modified local DG method, where the boundary conditions
are set so that the discrete Laplacian matrix is by construction
LU factored. The scheme was applied to several
standard Vlasov-Poisson test cases, which demonstrated the 
accuracy and robustness of the proposed scheme.

The advantage of DG over other methods that are based on larger
stencils is its ability to represent very rough data. 
We showed that this feature is important in the case of the
two-stream instability and the Landau damping calculations
presented in \S \ref{sec:numericaL_examples}. With an explicit time-stepping
method, the price that is paid for this spatial localization
is a maximum CFL number that decreases with
the spatial order of accuracy. In the context of Vlasov-Poisson we
have tamed this problem by using the semi-Lagrangian framework.

Future work will focus on extending the results described in this
paper to higher-dimensional Vlasov-Poisson equations. Furthermore,
modifications of the current approach to both the non-relativisitc
and the relativistic Vlasov-Maxwell equations will be considered.

\bigskip

\noindent
{\bf Acknowledgements.}
The authors would like to thank the anonymous reviewers for their
valuable feedback.
This work was supported in part by NSF grants DMS-0711885 and 
DMS-1016202.

\appendix

\section{Numerical evaluation of conserved quantities}
\label{sec:num_int}
The conserved quantities defined in \eqref{eqn:L1norm} ($L_1$-norm),
\eqref{eqn:L2norm} ($L_2$-norm), \eqref{eqn:energy} (total energy),
and \eqref{eqn:entropy} (entropy) are used as diagnostics of
the numerical methods proposed in this work. 

In order to evaluate all of these conserved quantities in the
numerical evolution, we define the following functional:
\begin{equation}
\label{eqn:int_functional}
I^h\bigl(g(f^h)\bigr) := \frac{\Delta x \, \Delta v}{4} \, \sum_{i=1}^{m_x}
	\sum_{j=1}^{m_v} \sum_{k=1}^{M^2} \omega_{k} \,
	g\left( f^h\left(x_i + \frac{\xi_k \Delta x}{2} , v_j + \frac{\eta_{k} \Delta v}{2}  \right) \right),
\end{equation}
where $m_x$ is the number of elements in the $x$-direction,
$m_v$ is the number of elements in the $v$-direction, and $\omega_k$  and $\left(\xi_k, \eta_k \right)$
are the $M^2$ Gauss-Legendre quadrature weights and points
on $[-1,1] \times [-1,1]$, respectively.
Expression \eqref{eqn:int_functional} gives a numerical approximation to integrals of the form:
\begin{equation}
   I\left(g(f)\right) := \int_{-L}^{L} \int_{-\infty}^{\infty} g\left(f(x,v)\right) \, dv \, dx.
\end{equation}
Using \eqref{eqn:int_functional} we define the following numerical
approximations to the norms defined by \eqref{eqn:L1norm}--\eqref{eqn:entropy}:
\begin{align}
\label{eqn:L1norm_num}
	\| f^h \|_{L_1} &:= I^h\left( \left| f^h \right| \right), \\
\label{eqn:L2norm_num}
	\| f^h \|_{L_2} &:= \left\{  \frac{\Delta x \, \Delta v}{4} \, \sum_{i=1}^{m_x}
	\sum_{j=1}^{m_v} \sum_{\ell=1}^{M(M+1)/2}   \left[ F^{(\ell)}_{ij} \right]^2
	 \right\}^{\frac{1}{2}}, \\	
\label{eqn:energy_num}
	\text{Total energy} &:=  \frac{1}{2} I^h\left( v^2 \, f^h \right) 		
		 +  \frac{\Delta x}{4} \sum_{i=1}^{m_x} \sum_{\ell=1}^M
		 	\left[ E^{(\ell)}_i \right]^2, \\
\label{eqn:entropy_num}
	\text{Entropy} &:= -I^h\left( f^h \, \log(f^h) \right).
\end{align}

\section{Relative $L_2$-norm error in 1D}
\label{sec:L2error_1D}
Let $f(x)$ be the exact solution of some problem of interest. 
Let $f^h(x)$ denote an approximation to $f(x)$ using a discontinuous
Galerkin method.
On each element $f^{h}(x)$ and $f(x)$ can be written as
\begin{align}
	f^{h}(x) \biggl|_{\Tm_i} &= \sum_{k=1}^M  F^{(k)}_i \, \varphi^{(k)}(\xi), \\
	f(x) \biggl|_{\Tm_i} &= \sum_{k=1}^{\infty} {\mathcal F}^{(k)}_i \, \varphi^{(k)}(\xi),
\end{align}
respectively.
The relative $L_2$-norm of the difference on the domain $x\in[a,b]$
between the approximation,
$f^{h}(x)$,
and the exact solution, $f(x)$, is given by
\begin{equation}
\begin{split}	
	 \frac{\| f(x) - f^{h}(x) \|_{L_2}}{\| f(x) \|_{L_2}} =& \, \left\{
			\frac{\int_a^b \left[ f(x)-f^{h}(x) \right]^2 \, dx}{
			 \int_a^b f(x)^2 \, dx} \right\}^\half \\
			 =& \, \left\{
		\frac{\sum_{i=1}^N   \sum_{k=1}^M   \left[
		 F^{(k)}_i   - {\mathcal F}^{(k)}_i \right]^2
		 }{\sum_{i=1}^N   \sum_{k=1}^M   
		\left[ {\mathcal F}^{(k)}_i  \right]^2}
		 \right\}^{\half} + {\mathcal O}\left( \Delta x^M \right),
\end{split}
\end{equation}
where $N$ is the total number of grid elements and $\Delta x = (b-a)/N$.
Therefore, we take as our relative $L_2$-norm indicator the following
easily computable quantity:
\begin{equation}
\label{eqn:L2error_1D}
E_2(\Delta x,M) := \left\{
		\frac{\sum_{i=1}^N   \sum_{k=1}^M   \left[
		 F^{(k)}_i   - {\mathcal F}^{(k)}_i \right]^2
		 }{\sum_{i=1}^N   \sum_{k=1}^M   
		\left[ {\mathcal F}^{(k)}_i  \right]^2}
		 \right\}^{\half}.
\end{equation}

\section{Relative $L_2$-norm error in 2D}
\label{sec:L2error_2D}
Let $f(x,y)$ be the exact solution of some problem of interest. 
Let $f^h(x,y)$ denote an approximation to $f(x,y)$ using a discontinuous
Galerkin method.
On each element $f^{h}(x,y)$ and $f(x,y)$ can be written as
\begin{align}
	f^{h}(x,y) \biggl|_{\Tm_{ij}} &= \sum_{k=1}^{M(M+1)/2}
		  F^{(k)}_{ij} \, \varphi^{(k)}(\xi,\eta), \\
	f(x,y) \biggl|_{\Tm_{ij}} &= \sum_{k=1}^{\infty} {\mathcal F}^{(k)}_{ij} \, \varphi^{(k)}(\xi,\eta),
\end{align}
respectively.
The relative $L_2$-norm of the difference on the domain $(x,y)\in \left[a_x,b_x \right] \times 
\left[a_y, b_y \right]$
between the approximation,
$f^{h}(x,y)$,
and the exact solution, $f(x,y)$, is given by
\begin{equation}
\begin{split}
		& \frac{\| f(x,y) - f^{h}(x,y) \|_{L_2}}{\| f(x,y) \|_{L_2}} 
		 = \, \left\{
			\frac{\int_{a_x}^{b_x} \int_{a_y}^{b_y} \left[ f(x,y)-f^{h}(x,y) \right]^2 \, dy \, dx}{
			 \int_{a_x}^{b_x} \int_{a_y}^{b_y} f(x,y)^2 \, dy \, dx} \right\}^\half \\
		&\quad = \, \left\{
		\frac{\sum_{i=1}^{N_x} \sum_{j=1}^{N_y}   \sum_{k=1}^{M(M+1)/2}   \left[
		 F^{(k)}_{ij}  - {\mathcal F}^{(k)}_{ij} \right]^2
		 }{\sum_{i=1}^{N_x} \sum_{j=1}^{N_y}   \sum_{k=1}^{M(M+1)/2}   \left[
		 {\mathcal F}^{(k)}_{ij} \right]^2}
		 \right\}^{\half} + {\mathcal O}\left(\Delta x^M, \Delta y^M \right),
\end{split}
\end{equation}
where $N_x$ and $N_y$ are the the number of grid elements in each coordinate
direction, $\Delta x = (b_x-a_x)/N_x$, and $\Delta y = (b_y-a_y)/N_y$.
Therefore, we take as our relative $L_2$-norm indicator the following
easily computable quantity:
\begin{equation}
\label{eqn:L2error_2D}
  E_2(\Delta x, \Delta y,M) := \left\{
		\frac{\sum_{i=1}^{N_x} \sum_{j=1}^{N_y}   \sum_{k=1}^{M(M+1)/2}   \left[
		 F^{(k)}_{ij}  - {\mathcal F}^{(k)}_{ij} \right]^2
		 }{\sum_{i=1}^{N_x} \sum_{j=1}^{N_y}   \sum_{k=1}^{M(M+1)/2}   \left[
		 {\mathcal F}^{(k)}_{ij} \right]^2}
		 \right\}^{\half}.
\end{equation}

\end{document}